\documentclass[oneside]{amsart}
\pdfoutput=1

\usepackage{geometry}
\usepackage[T1]{fontenc}
\usepackage[utf8]{inputenc}

\usepackage{amsmath,amssymb}
\usepackage{amsthm}
\usepackage{mathtools}
\usepackage{tikz}
\usetikzlibrary{positioning}

\usepackage{pgfplots}

\usepackage[pdfauthor={Michael Borinsky, Karen Vogtmann},
            pdftitle={The Euler characteristic of Out(Fn)},
            pdfsubject={Euler characteristic of the group of outer automorphisms of the free group},
            pdfkeywords={geometric group theory, group cohomology, Euler characteristic, groups of automorphisms, Outer space, group invariants, asymptotics, topological quantum field theory}]{hyperref}

\makeatletter
\newtheorem*{rep@theorem}{\rep@title}
\newcommand{\newreptheorem}[2]{%
\newenvironment{rep#1}[1]{%
 \def\rep@title{#2 \ref{##1}}%
 \begin{rep@theorem}}%
 {\end{rep@theorem}}}
\makeatother

\newtheorem{theorem}{Theorem}[section]%
\newreptheorem{theorem}{Theorem}
\newtheorem{thmx}{Theorem}
\newreptheorem{thmx}{Theorem}

\newtheorem{lemma}[theorem]{Lemma}
\newreptheorem{lemma}{Lemma}
\newtheorem{corollary}[theorem]{Corollary}
\newtheorem{proposition}[theorem]{Proposition}

\newtheorem{notation}[theorem]{Notation}

\theoremstyle{definition}
\newtheorem{definition}[theorem]{Definition}
\newtheorem{example}[theorem]{Example}

\theoremstyle{remark}
\newtheorem{remark}[theorem]{Remark}

\newcommand{\R}{\mathbb{R}}
\newcommand{\Q}{\mathbb{Q}}
\newcommand{\N}{\mathbb{N}}
\newcommand{\C}{\mathbb{C}}
\newcommand{\Z}{\mathbb{Z}}

\newcommand{\HH}{{\mathbf{H}}}
\newcommand{\HG}{{\mathbf{G}}}
\newcommand{\one}{\mathbb{I}}

\newcommand{\counit}{\epsilon}
\newcommand{\unit}{u}

\newcommand{\bigO}{\mathcal{O}}
\newcommand{\smallO}{o}

\DeclareMathOperator{\id}{id}
\DeclareMathOperator{\Aut}{Aut}
\DeclareMathOperator{\Out}{Out}

\newcommand{\ch}{\chi} %
\newcommand{\Ch}{\widehat\chi} %

\newcommand{\Outn}{\Out(F_n)}

\newcommand\G{\Gamma}
\newcommand{\GG}{{\mathcal{G}}} %
\newcommand{\lG}{\ell\mathcal{G}}  %
\newcommand{\LG}{\mathcal{LG}} %

\newcommand{\iso}{\cong}
\DeclareMathOperator{\PAut}{PAut}

\newlength\figureheight
\newlength\figurewidth
\title{The Euler characteristic of $\Outn$}
\author{Michael Borinsky \and Karen Vogtmann}

\address{
Michael Borinsky\\
Nikhef\\
Science Park 105\\
1098 XG Amsterdam\\
Netherlands
}
\address{
Karen Vogtmann\\
University of Warwick\\
Mathematics Institute\\
Zeeman Building\\
Coventry CV4 7AL\\
United Kingdom
}
\address{
Nikhef preprint number: 2019-031
}

\begin{document}

\begin{abstract} We prove that the rational Euler characteristic of $\Out(F_n)$ is always negative and  its asymptotic growth rate is $\G(n-\frac{3}{2})/\sqrt{2\pi}\log^2n$.  This settles a 1987 conjecture of J.\ Smillie and the second author.  We establish  connections with the Lambert $W$-function and the zeta function.
\end{abstract}

\maketitle

\thispagestyle{empty}
\section{Introduction}
The Euler characteristic, defined as the alternating sum of the Betti numbers,  is a key invariant of topological spaces of finite type (such as cell complexes built out of a finite number of cells).  One can define  an invariant $\widetilde \chi(G)$  for a group $G$ by substituting group cohomology for  singular cohomology, but unless $G$ has a finite-type $K(G,1)$-space this invariant lacks many desirable features of topological Euler characteristics.  This is unfortunate because many of the most interesting groups have torsion, and groups with torsion never have finite-type $K(G,1)$-spaces.  A solution was proposed by C.T.C.\ Wall,  who observed that if $G$ has a torsion-free subgroup $H$ of finite index that {\it does} have a finite-type classifying space then the rational number $\chi(H)/[G:H]$ is  an invariant of $G$ \cite{Wall61}. In particular, this number does not depend on the choice of $H$. Wall called it the {\it rational Euler characteristic} of $G$,   denoted    $\chi(G)$. This rational Euler characteristic is better behaved than $\widetilde \chi(G)$; for example if $1\to A\to B\to C\to 1$  is a short exact sequence of groups  then $\chi(B)=\chi(A)\chi(C)$, assuming  $\chi(A), \chi(B) $ and $\chi(C)$ are all defined.  

It turns out that rational Euler characteristics of arithmetic groups can often be expressed in terms of zeta functions; this ultimately depends on the Gauss-Bonnet theorem  (see \cite{Serre71, Serre79}  for details and a guide to the literature).  Remarkably,  Harer and Zagier showed that the rational Euler characteristics of mapping class groups are also given by zeta functions, e.g.\ the rational Euler characteristic of the mapping class group of a once-punctured surface of genus $g$ is equal to $\zeta(1-2g)$ \cite{HZ86}. This was later reproved by Penner \cite{penner1986moduli} and by Kontsevich \cite{Kon92}, each using asymptotic methods related to perturbation expansions from quantum field theory. 

 We are concerned here with the rational Euler characteristic of the group $\Outn$ of outer automorphisms of a finitely-generated free group.  This group   shares many features with both mapping class groups and arithmetic groups, though it belongs to neither class.  In 1987 Smillie and Vogtmann found a generating function for $\chi(\Outn)$ and computed its value for $n\leq 11$ \cite{SV87a}.  From the results of these computations, they conjectured that $\chi(\Outn)$ is always  negative and the absolute value grows faster than exponentially.  In 1989 Zagier simplified the generating function and computed $\chi(\Outn)$ for $n\leq 100$;  this added strong evidence for these conjectures without providing a proof \cite{Zagier}.    The only general statements previously known about the value of $\chi(\Outn)$  are that it is non-zero for even values of $n$, and certain information was established about the primes dividing the denominator \cite{SV87a, SV87b}.

In this article we show that $\chi(\Outn)$ is negative for all $n$ and prove that its asymptotic growth rate is controlled by the classical gamma and log functions:

\begin{thmx}\label{thm:SVconj}%
    The rational Euler characteristic of $\Outn$ is strictly negative, $\chi\left( \Outn \right) < 0$, for all $n \geq 2$ and its magnitude grows more than exponentially,
    \begin{align*}
        \chi\left(\Outn \right) &\sim - \frac{1}{\sqrt{2 \pi}} \frac{\Gamma(n-\frac32)}{\log^2 n} \text{ as } n \to \infty.
    \end{align*}
\end{thmx}

The proof of Theorem~\ref{thm:SVconj} is based on the following theorem, in which we  produce  an asymptotic expansion, with respect to the asymptotic scale  $\{(-1)^k\G(n+\frac{1}{2}-k)\}_{k\geq0}$ in the limit $n\to \infty$, whose coefficients are closely related to $\chi(\Outn)$.

\begin{thmx}\label{thm:asymptotic_expansion}
\begin{align*}
    \sqrt{2 \pi}e^{-n} n^n &\sim \sum_{ k\geq 0 } \Ch_k (-1)^k  \Gamma( n + \frac12 - k ) \text{ as } n\rightarrow \infty,
\end{align*}
where $\Ch_k$ is the coefficient of $z^k$ in the formal power series $\exp\left( \sum_{n\geq 1} \chi\left( \Out (F_{n+1}) \right) z^n \right)$.
\end{thmx}

We then relate this to a certain expansion of the Lambert $W$-function. The Lambert $W$-function is a well-studied function with a long history \cite{corless1996lambertw}. Eventually, we are able to use results of Volkmer  \cite{volkmer2008factorial} about the coefficients of this second expansion to prove Theorem~\ref{thm:SVconj}. 
In  Proposition~\ref{prop:efficient_chn} we also exploit the connection with the Lambert $W$-function to  give an efficient recursive algorithm for computing $\chi(\Out(F_n))$.

In Section~\ref{sec:core} we show that there is a close relationship between $\chi(\Outn)$ and the classical zeta function by considering the Connes-Kreimer Hopf algebra $\HH$ of 1-particle-irreducible graphs, i.e.\ graphs with no  separating edges.   Briefly, the formula in \cite{SV87a} for $\chi(\Outn)$ can be regarded as the integral of a certain character $\tau$ of $\HH$ on  the space spanned by admissible connected graphs with fundamental group $F_n$, with respect to the ``measure'' $\mu(\G)=1/|\Aut(\G)|$.  Proposition~\ref{prop:bernoulli_graphs_sum} shows that integrating $\tau^{-1}$ (the inverse  of $\tau$ in the group of characters) with respect to the same measure produces $\zeta(-n)/n=\frac{B_{n+1}}{n(n+1)},$ where $B_n$ is the $n$-th Bernoulli number.

The asymptotic expansion in  Theorem~\ref{thm:asymptotic_expansion} is strikingly  reminiscent of the well-known Stirling asymptotic expansion of the gamma function  in the asymptotic scale $\{\sqrt{2\pi} e^{-n} n^{n-\frac12-k} \}_{k\geq 0}$,
\begin{align*}
    \Gamma(n) &\sim 
    \sum_{k\geq 0} \widehat b_k \sqrt{2\pi} e^{-n} n^{n-\frac12-k} \text{ as } n \to \infty.
\end{align*}
The coefficients of this asymptotic expansion are related to the Bernoulli numbers as well: $\widehat b_k$ is the coefficient of $z^k$ in the formal power series $\exp\left( \sum_{n \geq 1} \frac{B_{n+1}}{n(n+1)} z^n \right)$. We will explore this analogy in more detail in Section~\ref{sec:asymptotic_expansions}.
Given this intriguing parallel between   the numbers $\chi(\Out(F_n))$ and  the Bernoulli numbers, which are very prominent objects in number theory, it would be interesting to look for a number theoretic meaning for the numbers $\chi(\Out(F_n))$ as well. The algorithm 
given in Proposition~\ref{prop:efficient_chn} may be helpful for investigations in this direction. 

As was pointed out in \cite{SV87a}, non-vanishing of $\chi(\Outn)$ implies that the kernel of the natural map from $\Outn$ to $GL(n,\Z)$ does not have finitely-generated homology.  Magnus proved in  1935 that this kernel is finitely generated and asked whether it is finitely presented, which would imply that the second homology is finitely generated \cite{Magnus}.   Bestvina, Bux and Margalit showed in 2007 that the homology in dimension $2n-4$ is not finitely generated \cite{BBM}. However Magnus' question is still unanswered for $n>3$.

Theorem~\ref{thm:SVconj} implies that for large $n,$   torsion-free subgroups of finite index in $\Out(F_n)$ have a huge amount of homology in odd dimensions.  We would like to say the same is true for the whole group $\Out(F_n)$.  One way to prove this is to compare the asymptotic growth rate of $\chi(\Out(F_n))$ with that of the ``naive'' Euler characteristic $\widetilde\chi(\Out(F_n))$. Brown  \cite{Brown82} showed that the difference between $\widetilde \chi$ and $\chi$ can be expressed in terms of rational Euler characteristics of centralizers of finite-order elements of $\Out(F_n)$.  Harer and Zagier used this to compare the rational and naive Euler characteristics for surface mapping class groups, using the fact that centralizers of finite-order elements are basically mapping class groups of surfaces of lower complexity.  Centralizers in $\Out(F_n)$ are more difficult to understand, but preliminary results obtained by combining the methods of this paper with results of Krsti\'{c} and Vogtmann \cite{KrVo93} indicate that the ratio $\widetilde\chi(\Out(F_n))/ \chi(\Out(F_n))$ tends to a positive constant.  Note that 
Galatius proved that  the stable rational homology of $\Outn$ vanishes  \cite{Galatius}, so this would show that there is a huge amount of {\em unstable} homology in odd dimensions.  This is completely mysterious, as only one odd-dimensional class has ever been found, by a very large computer calculation in rank $7$ \cite{Bartholdi}, and this calculation gives no insight into where all of these odd-dimensional classes might come from.

Finally we recall that, by the work of Kontsevich, the  cohomology of $\Outn$ with coefficients in a field of characteristic zero can be identified with the homology of the rank $n$ subcomplex of his Lie graph complex \cite{kontsevich1993formal}. 
Our results therefore apply to the Euler characteristic of this graph complex as well.  Kontsevich himself wrote down asymptotic formulas for the Euler characteristics of three of his graph complexes in \cite{kontsevich1993formal}; see Chapter 5 of  \cite{gerlits2002} for a detailed derivation of these formulas. The connection with graph complexes is explained a little further in Section~\ref{sec:kontsevich}. 
More discussion of the relation of the current paper with ideas from topological quantum field theory---with further relations to Kontsevich's work---can be found in Section~\ref{sec:tqft}.

\section*{Acknowledgements}
We thank Dirk Kreimer for support during this project. 
MB would like to thank Karen Yeats and Sam Yusim for discussions on the subject.
MB was supported by the NWO Vidi grant 680-47-551.
KV would like to thank Peter Smillie for discussions.  KV was partially supported by the Royal Society Wolfson Research Merit Award WM130098 and by a Humboldt Research Award.

\newcommand\vertex[1]{\fill #1 circle (.15)}
\newcommand{\A}{{\rm A}}
\section{Graphs and rational Euler characteristics}

In this section we give variations on the results of \cite{SV87a}.  The idea  is to use the action of $\Outn$ and closely related groups on contractible spaces of finite graphs
to deduce information about the homology of the groups, including the rational Euler characteristic.  

\subsection{Combinatorial graphs}

We begin with a combinatorial definition of a graph and related terms.

\begin{definition}  A  {\em  graph} $\G$  consists of  a finite set $H(\G)$ called {\em half-edges} together with
\begin{itemize}
\item a partition of $H(\G)$ into parts called {\em vertices} and
\item   an involution $\iota_\G\colon H(\G)\to H(\G)$.
\end{itemize}
The {\em valence} $|v|$  of a vertex $v$ is the number of half-edges in $v$. A {\em leaf} of $\G$ is a  fixed point of the involution $\iota_\G$ and an {\em edge}  is an   orbit of size $2$. 
A  {\em graph   isomorphism}  $\G\to\G'$ is a bijection  $H(\G)\to H(\G')$ that  preserves the vertex partitions and commutes with the involutions.  
\end{definition} 

\begin{notation} Let $\G$ be a graph.
\begin{itemize}
\item $\Aut(\G)$ is the group of isomorphisms   $\G\overset{\iso}{\rightarrow} \G$. 
\item   $L(\G)$ is the set of leaves of $\G$, and $s(\G)=|L(\G)|$.
\item $E(\G)$ is the set of  edges of $\G$ and  $e(\G)=|E(\G)|$.
\item $V(\G)$ is the set of  vertices of $\G$ and $v(\G)=|V(\G)|$.
\end{itemize}
\end{notation} 

The graph with one  vertex, $n$  edges, $s$ leaves and $2n+s$ half-edges is called a {\em thorned rose}, and will be denoted $R_{n,s}$.   If $n=0$ we will also call $R_{0,s}$ a {\em star graph}   (see Figure~\ref{fig:graphs}).

With the exception of Section~\ref{sec:core}, we will only consider {\em admissible} graphs, where a graph is admissible if all vertices have valence at least $3$.  

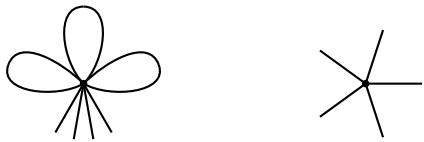
\begin{figure}
\begin{center}
 \begin{tikzpicture} [thick, scale=.25]   \fill (0,0) circle (.2);  \draw (0,0) to [out=135, in=180] (0,4.1); \draw (0,0) to [out=45, in=0] (0,4.1);  \draw (0,0) to [out=45, in=110] (4,1); \draw [thick] (0,0) to [out=-30, in=-70] (4,1);  \draw (0,0) to [out=135, in=70] (-4,1); \draw (0,0) to [out=210, in=-110] (-4,1);  \foreach \x in {-120, -100, -80, -60} { \draw (0,0) to (\x:3); }              \begin{scope} [xshift = 15cm]; \fill (0,0) circle (.2); \foreach \x in {0,72,144,216,288} { \draw (0,0) to (\x:3);  } \end{scope} \end{tikzpicture}
\end{center}
\caption{Thorned rose $R_{3,4}$ and star graph $R_{0,5}$}\label{fig:graphs}
\end{figure}
 
\begin{definition}  Let $\G$ be a graph.  A {\em subgraph} of $\G$ is a graph $\gamma$ with $H(\gamma)=H(\G)$, $V(\gamma)=V(\G)$, and  $E(\gamma)\subseteq E(\G)$.
\end{definition}
A graph $\G$ always has itself as a subgraph.  There is a unique ``trivial'' subgraph $\gamma_0$ with involution the identity, so $H(\gamma_0)=H(\G)$,$V(\gamma_0)=V(\G)$,  $E(\gamma_0)=\emptyset$ and $L(\gamma_0)=H(\G).$

\subsection{Topological graphs}  Every combinatorial graph $\G$ has a {\em topological realization} as a 1-dim\-en\-sional $CW$-complex.  For each element of $V(\G)$ we have a $0$-cell called a {\em vertex} and for each element of $E(\G)$ we have a $1$-cell called an {\em edge}.  For each element of  $L(\G)$ we have both a $0$-cell, called a {\em leaf vertex} and a $1$-cell connecting the leaf vertex to a  (non-leaf)  vertex.  
By our definition each connected component of the topological realization must have at least one vertex.  
Note that graphs may have multiple edges and  they may have loops at a vertex.  
The thorned rose $R_{n,s}$ defined in the last section has $(s+1)$ $0$-cells and $(n+s)$ 1-cells.   

The valence of a point $x$ is the minimal number of components of a deleted neighborhood of $x$.  
In an admissible graph the vertices are at least trivalent and the leaf vertices are univalent, so there are no bivalent $0$-cells.  

A graph isomorphism is a cellular homeomorphism, up to isotopy. Since admissible graphs have no bivalent  $0$-cells, any homeomorphism is a cellular homeomorphism. 

Notice that, by our definition, the topological realization of a subgraph of $\G$ is not a subcomplex of the topological realization of $\G$. Rather, it can be described as a graph obtained from $\G$ by cutting some of its edges, thus forming pairs of leaves. To make the result a $CW$-complex we have to add $0$-cells (leaf vertices) to the ends of the cut edges. A subgraph can also be visualized as the closure of a sufficiently small neighborhood of a subcomplex.  
 
 In  the remainder of this section we will work with the topological realization of a graph instead of using  the combinatorial definition, so that we may freely use topological concepts such as connectivity, fundamental group and homotopy equivalence.

For $s\geq 0$ let 
\begin{itemize}
\item $\GG$ denote the set of isomorphism classes of finite admissible graphs,
\item $\GG_s\subset \GG$ be the subset consisting of admissible graphs with exactly $s$ leaves, 
\item $\GG^c \subset \GG$ and $\GG_s^c \subset \GG_s$ be the respective subsets of connected graphs. 
\end{itemize} 
 
\subsection{Groups}
\label{sec:groups}

For any $n$ and $s$ we define   $\A_{n,s}$  to be the group of homotopy classes of homotopy equivalences of the thorned rose $R_{n,s}$ that fix the leaf vertices  $\{b_1,\ldots,b_s\},$ i.e.\ $\A_{n,s}=\pi_0(HE(\G, b_1,\ldots,b_s))$.     The groups $\A_{n,s}$ generalize $\Outn \iso \A_{n,0}$ and $\Aut(F_n) \iso \A_{n,1}$.  If $n=0$ then $R_{0,s}$ is a graph with no loops and at least $3$ leaves as we are insisting on at least one vertex, which is at least trivalent. So, $\A_{0,s}$ is only defined for $s\geq 3$, where it is the trivial group.    If $n=1$ then $R_{1,s}$ is a loop with $s\geq 1$ leaves, and there is a short exact sequence  $$1\to \Z^{s-1}\to \A_{1,s}\to \Z/2\Z\to 1.$$  
 If $n\geq 2$ and $s\geq 0$ there is a short exact sequence $$1\to F_{n}^{s}\to \A_{n,s}\to \Out(F_{n})\to 1.$$
 See \cite{CHKV16} for background on the groups $\A_{n,s}$.
These groups appear, for example, in the context of homology stability theorems \cite{Hat95}, the bordification of Outer space and virtual duality \cite{BF00, BSV18}, and assembly maps for homology \cite{CHKV16}.    
 
\subsection{Complexes of graphs and the rational Euler characteristic}
If a group $G$ is virtually torsion-free and acts properly and cocompactly on a contractible $CW$-complex $X$, then the rational Euler characteristic $\chi(G)$ can be calculated using this action, by the formula $$\chi(G)=\sum_{\sigma\in \mathcal C} \frac{(-1)^{\dim \sigma}}{|\text{Stab}(\sigma)|},$$
where $\mathcal C$ is a set of representatives for the cells of $X$ mod $G$ (see, e.g.\ \cite{Bro82}, Proposition~(7.3)).

 For any $s\geq 0$ the group $\A_{n,s}$ is virtually torsion-free and acts  properly and cocompactly on a contractible cube complex $K_{n,s}$.  To describe $K_{n,s}$ it is convenient to label the leaves of a graph, so that two graphs $\G$ and $\G'$ are isomorphic if there is a  graph isomorphism $\G\to\G'$ that preserves leaf-labels; an isomorphism class is then called a {\em leaf-labeled graph}. We use the notation $\lG$, $\lG_s$, $\lG^c$ and $\lG^c_s$ instead of $\GG$, $\GG_s$, $\GG^c$ and $\GG^c_s$ to denote the respective set of leaf-labeled graphs and $\PAut(\G)$ to denote the set of  automorphisms of a graph that fix the leaves.
 
 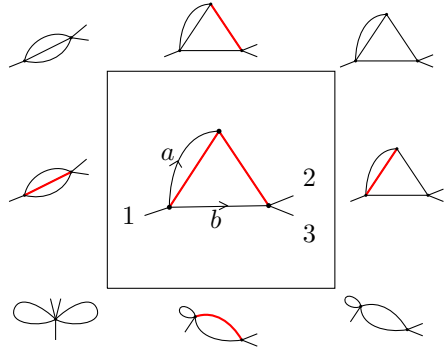
\begin{figure}
 \begin{center}
 \begin{tikzpicture} [scale=.55]   \draw (6.25,-1.75) to (6.25,3.5) to (11.75,3.5) to (11.75,-1.75) to (6.25,-1.75);  \begin{scope}[xshift=7.75cm, yshift=.25cm, scale=.4] \coordinate (a) at (0,-.1); \coordinate (b) at (3,4.5); \coordinate (c) at (6,0); \coordinate (one) at (-1.5,-.6); \coordinate (two) at (7.5,-.6); \coordinate (three) at (7.5,.6); \draw (a) to (b) to (c) to (a); \draw (a) to [out=90, in= 180 ] (b); \draw[thick, red] (a) to (b) to (c); \draw (a) to (one);  \node [left] (x) at (one) {$1$}; \draw (c) to (two);  \node [below right] (x) at (two) {$3$}; \draw (c) to (three);  \node [above right] (x) at (three) {$2$};  \node (x) at (-.1,3.1) {$a$}; \node (y) at (2.8,-.8) {$b$}; \draw (0,2.4) to (.5,2.7) to (.75,2.3); \draw (3,-.25) to (3.5,0) to (3,.25); \vertex{(a)};\vertex{(b)};\vertex{(c)}; \end{scope}  \begin{scope}[xshift=12.25cm, yshift=3.75cm,scale=.25] \coordinate (a) at (0,0); \coordinate (b) at (3,4.5); \coordinate (c) at (6,0); \draw (a) to (b) to (c) to (a); \draw (a) to [out=90, in= 180 ] (b); \draw (a) to (-1.5,-.6); \draw (c) to (7.5,-.6);\draw (c) to (7.5,.6); \vertex{(a)};\vertex{(b)};\vertex{(c)}; \end{scope}  \begin{scope}[xshift=12.5cm,yshift=.5cm, scale=.25] \coordinate (a) at (0,0); \coordinate (b) at (3,4.5); \coordinate (c) at (6,0); \draw (a) to (b) to (c) to (a); \draw (a) to [out=90, in= 180 ] (b); \draw[thick, red] (b) to (a); \vertex{(a)};\vertex{(b)};\vertex{(c)}; \draw (a) to (-1.5,-.6); \draw (c) to (7.5,-.6); \draw (c) to (7.5,.6); \end{scope}  \begin{scope}[xshift=8cm, yshift=4cm,scale=.25] \coordinate (a) at (0,0); \coordinate (b) at (3,4.5); \coordinate (c) at (6,0); \draw (a) to (b) to (c) to (a); \draw (a) to [out=90, in= 180 ] (b); \draw[thick, red] (b) to (c); \vertex{(a)};\vertex{(b)};\vertex{(c)}; \draw (a) to (-1.5,-.6); \draw (c) to (7.5,-.6); \draw (c) to (7.5,.6); \end{scope}  \begin{scope}[xshift=12cm, yshift=-2.75cm, scale=.25] \coordinate (c) at (6,0); \coordinate (n) at (1.5,2.25); \coordinate(r) at (0,3); \vertex{(c)};\vertex{(n)}; \draw (n) to[out=90, in=70] (r); \draw (n) to [out=210, in=-110] (r); \draw (c) to [out=190, in= -90] (n);\draw (c) to [out=120, in= 20 ](n); \draw (n) to (.5,.75);\draw (c) to (7.5,-.6); \draw (c) to (7.5,.6); \end{scope}  \begin{scope}[xshift=8cm, yshift=-3cm, scale=.25] \coordinate (c) at (6,0); \coordinate (n) at (1.5,2.25); \coordinate(r) at (0,3); \draw (n) to[out=90, in=70] (r); \draw (n) to [out=210, in=-110] (r); \draw (c) to [out=190, in= -90] (n); \draw [thick, red] (c) to [out=120, in= 20 ](n); \vertex{(c)};\vertex{(n)}; \draw (n) to (.5,.75); \draw (c) to (7.5,-.6); \draw (c) to (7.5,.6); \end{scope}  \begin{scope}[xshift=4.25cm, yshift=.5cm, scale=.25] \coordinate (a) at (0,0);\coordinate (m) at (4.5,2.25); \vertex{(a)};\vertex{(m)}; \draw (a) to (m); \draw[thick, red] (a) to (m); \draw (a) to [out=80, in= 150] (m); \draw (a) to [out=-15, in= -105 ](m); \draw (a) to (-1.5,-.6); \draw (m) to (6,3.5); \draw (m) to (6.2,2); \end{scope}  \begin{scope}[xshift=4.25cm, yshift=3.75cm, scale=.25] \coordinate (a) at (0,0); \coordinate (b) at (3,4.5); \coordinate (c) at (6,0); \coordinate (m) at (4.5,2.25); \coordinate (n) at (1.5,2.25); \vertex{(a)};\vertex{(m)}; \draw (a) to (m); \draw (a) to [out=80, in= 150] (m); \draw (a) to [out=-15, in= -105 ](m); \draw (a) to (-1.5,-.6); \draw (m) to (6,3.5); \draw (m) to (6.2,2); \end{scope}  \begin{scope}[xshift=5cm, yshift=-2.5cm, scale=.25] \vertex{(0,0)};  \draw (0,0) to [out=45, in=110] (4,1); \draw (0,0) to [out=-30, in=-70] (4,1);  \draw (0,0) to [out=135, in=70] (-4,1); \draw (0,0) to [out=210, in=-110] (-4,1);  \draw (0,0) to (0,-2); \draw (0,0) to (-.5,2); \draw (0,0) to (.5,2); \end{scope} \end{tikzpicture}
      \caption{A $2$-dimensional cube $(\G,\varphi)$ in $K_{2,3}$. Here $\varphi$ is the red subgraph, $\pi_1(\G)=F\langle a, b\rangle$ and the leaves are labeled $1,2,3$. The marking is indicated  by  arrows and labels on the edges in the complement of a maximal tree.}\label{fig:cube}
      \end{center}
      \end{figure}

The cube complex $K_{n,s}$ has one cube for each equivalence class of triples $(\G, \varphi, g)$, where 
\begin{itemize}
    \item $\G\in \lG^c_s$ is connected with $s$ labeled leaves and with $\pi_1(\G) \cong F_n$,
    \item $\varphi$ is a \textit{subforest} of $\G,$ i.e.\ a subgraph with no cycles,  
    \item $g\colon R_{n,s}\to \G$ is a leaf-label-preserving homotopy equivalence, called a \textit{marking} and
    \item $(\G,\varphi,g)\sim (\G',\varphi',g')$ if there is a  leaf-label-preserving graph isomorphism $h\colon\G\to\G'$ sending $\varphi$ to $\varphi'$ such that $h\circ g$ is homotopic to $g'$ through  leaf-label-preserving homotopies.
\end{itemize}
An example of a cube in $K_{2,3}$ is depicted in Figure~\ref{fig:cube}.
Contractibility of $K_{n,s}$ was proved for $s=0, n\geq 2$ by Culler and Vogtmann \cite{CV86} and in general by Hatcher \cite{Hat95} (see also \cite{HV96}).  (For $n\geq 2$, $K_{n,s}$ was originally described as a simplicial complex, but its simplices naturally group themselves into cubes, as was done, e.g.\  in \cite{HV98}.)   

 Smillie and  Vogtmann \cite{SV87a} considered only the case $s=0,$ but their arguments apply verbatim for graphs with leaves.  We define a function
 \[
\tau(\G)=\sum_{\varphi\subset \G} (-1)^{e(\varphi)},
\] 
where the sum is over all subforests $\varphi\subset \G,$ including the trivial subgraph, and $e(\varphi)$ is the number of edges in $\varphi$.  
For instance, 
$\tau(
{
\begin{tikzpicture}[x=1ex,y=1ex,baseline={([yshift=-.6ex]current bounding box.center)}] \coordinate (vm); \coordinate [left=.7 of vm] (v0); \coordinate [right=.7 of vm] (v1); \draw (v0) circle(.7); \draw (v1) circle(.7); \filldraw (vm) circle (1pt); \end{tikzpicture}%
}
) = 1$, 
as it  has only the trivial subforest and
$\tau(
{
\begin{tikzpicture}[x=1ex,y=1ex,baseline={([yshift=-.6ex]current bounding box.center)}] \coordinate (vm); \draw (vm) to (-.3,-1.25);\draw (vm) to (.3,-1.25); \coordinate [left=.7 of vm] (v0); \coordinate [right=.7 of vm] (v1); \draw (v0) circle(.7); \draw (v1) circle(.7); \filldraw (vm) circle (1pt); \end{tikzpicture}%
}
) = 1$ for the same reason (recall that a leaf is not an edge).
We have $\tau(
{
\begin{tikzpicture}[x=1ex,y=1ex,baseline={([yshift=-.6ex]current bounding box.center)}] \coordinate (v0); \coordinate [right=1.5 of v0] (v1); \coordinate [left=.7 of v0] (i0); \coordinate [right=.7 of v1] (o0); \draw (v0) -- (v1); \filldraw (v0) circle (1pt); \filldraw (v1) circle (1pt); \draw (i0) circle(.7); \draw (o0) circle(.7); \end{tikzpicture}%
}
) = 0$, as it has two forests whose respective contributions to the sum cancel (in fact $\tau$ always vanishes on graphs with a separating edge as ensured by Lemma~(2.3) of \cite{SV87a}) and 
$\tau(
{
\begin{tikzpicture}[x=1ex,y=1ex,baseline={([yshift=-.6ex]current bounding box.center)}] \coordinate (vm); \coordinate [left=1 of vm] (v0); \coordinate [right=1 of vm] (v1); \draw (v0) -- (v1); \draw (vm) circle(1); \filldraw (v0) circle (1pt); \filldraw (v1) circle (1pt); \end{tikzpicture}%
}
)=-2
$, as the graph $
{
\begin{tikzpicture}[x=1ex,y=1ex,baseline={([yshift=-.6ex]current bounding box.center)}] \coordinate (vm); \coordinate [left=1 of vm] (v0); \coordinate [right=1 of vm] (v1); \draw (v0) -- (v1); \draw (vm) circle(1); \filldraw (v0) circle (1pt); \filldraw (v1) circle (1pt); \end{tikzpicture}%
}$ has four forests with contributions $-1,-1,-1$ and $1$.

The following theorem relates the rational Euler characteristics  of the groups $\A_{n,s}$ to the function $\tau$.   
 \begin{theorem}\label{thm:SV}%
 \[
     \chi(\A_{n,s})=\sum \frac{\tau(\G)}{|\PAut(\G)|},
\]
where the sum is over all connected leaf-labeled graphs $\Gamma$ with $s$ leaves and fundamental group $F_n$. \end{theorem} 
 \begin{proof}  
     The group $\A_{n,s}$ acts properly and cocompactly on $K_{n,s}$.  It acts transitively on markings, assuming the graphs are leaf-labeled, so there is one orbit of cubes for each isomorphism class of pairs $(\G,\varphi)\in\lG^c_s$ with fundamental group $F_n$.  The dimension of this cube is $e(\varphi)$, and the stabilizer of  $(\G,\varphi, g)$ is isomorphic to the group $\PAut(\G,\varphi)$ of  automorphisms of $\G$ that fix the leaves and send $\varphi$ to itself.  Therefore we have
 $$\chi(\A_{n,s})=\sum_{\sigma\in \mathcal C} \frac{(-1)^{\dim \sigma}}{|\text{Stab}(\sigma)|}=\sum_{(\G,\varphi)} \frac{(-1)^{e(\varphi)}}{|\PAut(\G,\varphi)|},$$
 where the sum is over all isomorphism classes of pairs $(\G,\varphi)$ of leaf-labeled graphs $\G\in \lG_s^c$ and forests $\varphi\subset \G$.   The full automorphism group $\PAut(\G)$ acts on the set of forests in $\G$, and an orbit is an isomorphism class of pairs $(\G,\varphi)$, so  the orbit-stabilizer theorem now gives
  $$\sum_{(\G,\varphi)} \frac{(-1)^{e(\varphi)}}{|\PAut(\G,\varphi)|} = \sum_{\G\in\lG^c_s} \sum_{ \varphi\subset \G} \frac{(-1)^{e(\varphi)}}{|\PAut(\G)|}=\sum_{\G\in\lG^c_s} \frac{\tau(\G)}{|\PAut(\G)|}$$
  as desired.

 Note that for $n\geq 2$ and $s=0$ we have $\lG_0^c=\GG_0^c,$  $\A_{n,0} \iso \Outn$ and $\PAut(\G) \iso \Aut(\G)$, so  this is  Proposition~(1.12) of   \cite{SV87a}.   
 \end{proof}

\begin{example} 
Using this theorem we can immediately verify that
\begin{gather*}
\chi(\A_{2,0}) = \chi(\Out(F_2))= \\
\frac{\tau(
{
\begin{tikzpicture}[x=1ex,y=1ex,baseline={([yshift=-.6ex]current bounding box.center)}] \coordinate (vm); \coordinate [left=.7 of vm] (v0); \coordinate [right=.7 of vm] (v1); \draw (v0) circle(.7); \draw (v1) circle(.7); \filldraw (vm) circle (1pt); \end{tikzpicture}%
}
)}{|\PAut(
{
\begin{tikzpicture}[x=1ex,y=1ex,baseline={([yshift=-.6ex]current bounding box.center)}] \coordinate (vm); \coordinate [left=.7 of vm] (v0); \coordinate [right=.7 of vm] (v1); \draw (v0) circle(.7); \draw (v1) circle(.7); \filldraw (vm) circle (1pt); \end{tikzpicture}%
}
)|}
+
\frac{\tau(
{
\begin{tikzpicture}[x=1ex,y=1ex,baseline={([yshift=-.6ex]current bounding box.center)}] \coordinate (v0); \coordinate [right=1.5 of v0] (v1); \coordinate [left=.7 of v0] (i0); \coordinate [right=.7 of v1] (o0); \draw (v0) -- (v1); \filldraw (v0) circle (1pt); \filldraw (v1) circle (1pt); \draw (i0) circle(.7); \draw (o0) circle(.7); \end{tikzpicture}%
}
)}{|\PAut(
{
\begin{tikzpicture}[x=1ex,y=1ex,baseline={([yshift=-.6ex]current bounding box.center)}] \coordinate (v0); \coordinate [right=1.5 of v0] (v1); \coordinate [left=.7 of v0] (i0); \coordinate [right=.7 of v1] (o0); \draw (v0) -- (v1); \filldraw (v0) circle (1pt); \filldraw (v1) circle (1pt); \draw (i0) circle(.7); \draw (o0) circle(.7); \end{tikzpicture}%
}
)|}
+
\frac{\tau(
{
\begin{tikzpicture}[x=1ex,y=1ex,baseline={([yshift=-.6ex]current bounding box.center)}] \coordinate (vm); \coordinate [left=1 of vm] (v0); \coordinate [right=1 of vm] (v1); \draw (v0) -- (v1); \draw (vm) circle(1); \filldraw (v0) circle (1pt); \filldraw (v1) circle (1pt); \end{tikzpicture}%
}
)}{|\PAut(
{
\begin{tikzpicture}[x=1ex,y=1ex,baseline={([yshift=-.6ex]current bounding box.center)}] \coordinate (vm); \coordinate [left=1 of vm] (v0); \coordinate [right=1 of vm] (v1); \draw (v0) -- (v1); \draw (vm) circle(1); \filldraw (v0) circle (1pt); \filldraw (v1) circle (1pt); \end{tikzpicture}%
}
)|}
=
\frac{1}{8} + \frac{0}{8} + \frac{-2}{12} = 
-\frac{1}{24}.
\end{gather*}

\end{example}

 \begin{corollary}\label{cor:SV}
     $$ \frac{\chi(\A_{n,s})}{s!}= \sum_{\substack{\G\in\GG^c_{s}\\ \pi_1(\G) \cong F_n}} \frac{\tau(\G)}{|\Aut(\G)|}.$$
 \end{corollary}

 \begin{proof} 
     $\Aut(\G)$ acts on the set  of leaf-labelings of $\G \in \lG_s^c$. The orbits are the leaf-labeled graphs,  and the stabilizer of a labeling is $\PAut(\G)$, giving $|\Aut(\G)|=|\text{Orbit}(\G)| |\PAut(\G)|. $ There are $s!$ labelings of $\G.$ Each orbit has the same size,  so the size of each orbit is  $s!/\ell(\G)$, where $\ell(\G)$ is the number of leaf-labeled graphs with underlying graph $\G$. Therefore,
  \begin{align*}  
      |\Aut(\G)|&=\frac{s!}{\ell(\G)}  |\PAut(\G)|. \qedhere
  \end{align*}
 \end{proof}

\section{Formal power series} 
For the rest of this article it will be convenient to  use $|\G|=e(\G)-v(\G)$ instead of the rank of $\pi_1(\G)$ to filter the set of graphs. For connected graphs $\G$ this is only a minor change of notation as $\text{rank}(\pi_1(\G)) = e(\G)-v(\G) +1 =|\G|+1$. %
 
Consistent with this shift, we define $\ch_n = \chi(\Out(F_{n+1}))$ and consider the formal power series
\begin{align*}
T(z)=\sum_{n=1}^\infty \chi(\Out(F_{n+1})) z^n=\sum_{n=1}^\infty \ch_n z^n.
\end{align*}
By Theorem~\ref{thm:SV} with $s=0$ we have 
\begin{align}
\label{eqn:def_Tz_cntd_graph_sum}
T(z)=\sum_{n=1}^\infty \left(\sum_{\substack{\G\in\GG^c_0\\ \pi_1(\G) \cong F_{n+1}}}\frac{\tau(\G)}{|\Aut(\G)|}\right) z^n = \sum_{\G\in\GG^c_0}\frac{\tau(\G)}{|\Aut(\G)|}z^{|\G|}.
\end{align}
For general $A_{n,s}$ we define a bivariate generating function for the Euler characteristic by
\begin{align}
\label{eqn:def_Tzx_Ans}
T(z,x)= \sum_{s\geq 3} \ch(\A_{0,s})z^{-1}\frac{x^s}{s!} + \sum_{s\geq 1} \ch(\A_{1,s}) \frac{x^s}{s!} + \sum_{n\geq 1}\sum_{s\geq 0} \ch(\A_{n+1,s})z^n\frac{x^s}{s!}.
\end{align}
Recall that the groups $\A_{n,s}$ are only defined for $s \geq 3$ if $n=0$, for $s \geq 1$ if $n=1$ and for $s\geq 0$ if $n \geq 2$. 
Obviously, $T(z,0) = T(z)$ and by Corollary~\ref{cor:SV}
\begin{align}
\label{eqn:def_Tzx_cntd_graph_sum}
T(z,x)=\sum_{\G\in\GG^c}\frac{\tau(\G)}{|\Aut(\G)|}z^{|\G|}x^{s(\G)},
\end{align}
where $s(\G)$ is the number of leaves in $\G$.  

The relationships between the groups $\A_{n,s}$, which were described in Section~\ref{sec:groups}, imply the following functional relation between $T(z)$ and the bivariate generating function $T(z,x)$.
 
\begin{proposition}  
    \label{prop:Tzx_leaves_identity}
    $$T(z,x) = \frac{e^x-\frac{x^2}{2}-x-1}{z} + \frac{x}{2} + T(ze^{-x}).$$
\end{proposition}

\begin{proof} 
The groups $\A_{0,s}$ are trivial, so we have $\ch(\A_{0,s})=1$ for all $s \geq 3$.  
 For the groups $\A_{1,s}$ we have the short exact sequence $$1\to \Z^{s-1}\to \A_{1,s}\to \Z/2\Z\to 1.$$  
Thus $\chi(\A_{1,s})=0$ if $s\geq 2$ and $\chi(\A_{1,1})=\chi(\Z/2\Z)=\frac{1}{2}$.
For the groups $\A_{n+1,s}$ with $n\geq 1$ the short exact sequence $$1\to F_{n+1}^{s}\to \A_{n+1,s}\to \Out(F_{n+1})\to 1,$$
gives $\chi(\A_{n+1,s})=\chi(\Out(F_{n+1}))(-n)^{s}=\ch_n (-n)^{s}$. 

Putting these together into eq.~\eqref{eqn:def_Tzx_Ans} gives 
\begin{align*}
    T(z,x)&=\sum_{s\geq 3} \frac{x^s}{s!}z^{-1} + \frac{x}{2} +\sum_{n\geq 1}\sum_{s\geq 0} \ch_n\frac{ (-n)^s}{s!} z^nx^s \\
    & =\sum_{s\geq 3} \frac{x^s}{s!}z^{-1} + \frac{x}{2} +\sum_{n\geq 1}\ch_n \sum_{s\geq 0} \frac{(-n)^s}{s!} x^s  z^n \\
    & =\sum_{s\geq 3} \frac{x^s}{s!}z^{-1} + \frac{x}{2} +\sum_{n\geq 1}\ch_n e^{-nx} z^n\\
 &=\sum_{s\geq 3} \frac{x^s}{s!}z^{-1} + \frac{x}{2} + T(ze^{-x}). \qedhere
\end{align*}
\end{proof}

\section{Algebraic graph combinatorics}
\label{sec:graphical_enumeration}

Although Theorem~\ref{thm:SV} gives an explicit expression for the coefficients of $T(z)$ and $T(z,x)$, we will use an implicit equation, which the generating function $T(z,x)$ must satisfy, to prove Theorem \ref{thm:asymptotic_expansion}. This implicit equation together with the identity from Proposition~\ref{prop:Tzx_leaves_identity} determines the coefficients $\chi(\Out(F_n))$ completely.

To formulate this implicit equation, it is convenient to use the \textit{coefficient extraction operator} notation: For an arbitrary formal power series $f(x)$ the notation $[x^n] f(x)$ means `the $n$-th coefficient in $x$ of $f(x)$.'

\begin{proposition}
\label{prop:Tzx_graph_counting_identity}
\begin{align}
\label{eqn:Tzx_implicit_equation}
    1 &= \sum_{\ell \geq 0} (-z)^\ell (2\ell-1)!! [x^{2\ell}] \exp\left( T(z,x) \right),
\end{align}
where $(2\ell-1)!!=(2\ell)! / ( \ell! 2^\ell )$ is  the  double factorial.
\end{proposition}
In the remainder of this section we will first give a combinatorial reformulation of this identity and then prove it.

\subsection{The exponential formula}
The exponential of the generating function $T(z,x)$ in \eqref{eqn:Tzx_implicit_equation} has a straightforward combinatorial interpretation. While $T(z,x)$ can be expressed as a sum over connected graphs as we did in \eqref{eqn:def_Tzx_cntd_graph_sum}, the generating function $\exp(T(z,x))$ can be expressed as a sum over all graphs.
The reason for this is that the function $\tau$ is multiplicative on disjoint unions of graphs, so we can apply the \textit{exponential formula} given, for example, in  \cite[5.1]{stanley1997enumerative2}.

Briefly, the argument behind the exponential formula is that if $\phi$ is a function on graphs that is multiplicative on disjoint unions, i.e.\ $\phi(\G_1 \sqcup \G_2) = \phi(\G_1) \phi(\G_2)$, then 
\begin{align*}
\sum_{\substack{\G \in \GG \\ |C(\G)| = n}} \frac{\phi(\G)}{|\Aut\G|} = \frac{1}{n!} \sum_{\gamma_1, \ldots, \gamma_n \in \GG^c} \prod_{i=1}^n \frac{\phi(\gamma_i)}{|\Aut \gamma_i|},
\end{align*}
where we sum over all graphs with $n$ connected components on the left hand side and over all $n$ tuples of connected graphs on the right hand side. The factor $1/n!$ accounts for the number of ways to order the connected components of the graphs. If we sum this equation over all $n \geq 0$ and use $e^x = \sum_{n \geq 0} x^n/n!$, we get
\begin{lemma}[Exponential formula]
\label{lmm:exponential_formula}
Let $\phi$ be a function from graphs to a power series ring that is multiplicative on disjoint unions (i.e.\ $\phi(\G_1 \sqcup \G_2) = \phi(\G_1) \phi(\G_2)$).  If the coefficient in $\phi(\G)$ of a given monomial  is non-zero for only finitely many graphs $\G$, then  
\begin{align*}
\sum_{\G \in \GG} \frac{\phi(\G)}{|\Aut\G|}
= \exp\left( 
\sum_{\G \in \GG^c} \frac{\phi(\G)}{|\Aut\G|}
\right).
\end{align*}
\end{lemma}
The finiteness condition on the function $\phi$ is necessary to ensure that $\sum_{\G \in \GG} \frac{\phi(\G)}{|\Aut\G|}$ exists in the respective power series space.
 
\begin{corollary}
\label{crll:disconnected_exp}
\begin{align*}
\exp({T(z)}) &= \sum_{\G\in\GG_0}\frac{\tau(\G)}{|\Aut(\G)|}z^{|\G|}\\
\exp({T(z,x)}) &= \sum_{\G\in\GG}\frac{\tau(\G)}{|\Aut(\G)|}z^{|\G|}x^{s(\G)}.
\end{align*}
\end{corollary}

\begin{proof}
Let $\phi_1$ be the function defined by $\phi_1(\G) = \tau(\G) z^{|\G|}$ for $\G \in \GG_0$ and $\phi_1(\G)=0$ for $\G \in \GG_s$ with $s \geq 1$. This function is multiplicative on disjoint unions of graphs, because $\tau$ is (Lemma~(2.2) of \cite{SV87a}) and both $|\G|$ and $s(\G)$ are additive graph invariants. The first statement follows by applying Lemma~\ref{lmm:exponential_formula}  to $\phi_1$ and using  eq.\ \eqref{eqn:def_Tz_cntd_graph_sum}. For the second statement apply  Lemma~\ref{lmm:exponential_formula} to the function $\phi_2(\G) = \tau(\G) z^{|\G|} x^{s(\G)}$ for all $\G \in \GG$, note that there are only a finite number of admissible graphs with fixed $|\G|$ and $s(\G)$ and apply eq.\ \eqref{eqn:def_Tzx_cntd_graph_sum}. 
\end{proof}
Note that the power series $T(z)$ and $\exp(T(z))$ carry the same information. 
Recall that $\ch_n$ is the coefficient of $z^n$ in $T(z)$, and denote the  coefficient of $z^n$ in $\exp(T(z))$ by $\Ch_n$.  
The coefficients $\ch_n$ and $\Ch_n$ are related by
\begin{align}
\label{eqn:exp_pwrsrs_relation}
     \ch_n&=\Ch_n-\frac{1}{n}\sum_{k=1}^{n-1}k \ch_k \Ch_{n-k} \text{ for } n \geq 1.
\end{align}
This recursive relation can be derived by taking the formal derivative of $\exp(T(z))$ which results in the (formal) differential equation $\frac{d}{dz}\exp({T(z)}) = e^{T(z)} \frac{d}{dz}T(z)$. Note that it is also important that $\ch_0 = 0$ for $\exp(T(z))$ to make sense as a power series with $\Q$ as coefficient ring. 

We can immediately use the relationship between the coefficients $\Ch_n$ and $\ch_n$ to prove the following statement which will be helpful later while proving that the rational Euler characteristic of $\Out(F_n)$ is always negative.
\begin{lemma}
\label{lmm:exp_negative}
If $\Ch_n < 0$ for all $n \geq 1$, then $\ch_n < 0$ for all $n \geq 1$.
\end{lemma}
\begin{proof}
This follows by induction on $n$ on eq.\ \eqref{eqn:exp_pwrsrs_relation}. 
\end{proof}
Because $\ch_n = \chi(\Out(F_{n+1}))$, proving $\Ch_n < 0$ for all $n\geq 1$ is therefore sufficient to show that $\chi(\Out(F_n)) < 0$ for all $n\geq 2$.

\subsection{Convolution identities}

By Corollary~\ref{crll:disconnected_exp}, the statement of Proposition~\ref{prop:Tzx_graph_counting_identity} is equivalent to the identity
\begin{align}
\label{eqn:graph_sum_tau_identity}
1=\sum_{\ell\geq 0} (-z)^\ell(2\ell-1)!! \sum_{\G\in\GG_{2\ell}} \frac{\tau(\G)}{|\Aut(\G)|} z^{|\G|}.
\end{align}

If $\gamma$ is a subgraph of $\G$, we denote by  $\G/\gamma$  the graph that one obtains from $\G$ by collapsing each edge that is in $\gamma$ to a point.
We will use the following \textit{convolution identity} for $\tau$ to prove eq.\ \eqref{eqn:graph_sum_tau_identity} and therefore also Proposition~\ref{prop:Tzx_graph_counting_identity}. 
\begin{proposition}
\label{prop:tau_identity}
If $\G$ is a graph with at least one cycle, then 
\begin{align*}
   \sum_{\gamma \subset \G} \tau(\gamma) (-1)^{e(\G/\gamma)} = 0.
\end{align*}
where the sum is over all subgraphs of $\G$, including the trivial subgraph and $\G$ itself.  
\end{proposition}

 This statement can be seen as an identity in the incidence algebra of the subgraph poset of a graph. We will discuss a related viewpoint in Section~\ref{sec:core}, where we will interpret it as an inverse relation in the group of characters of the Hopf algebra of core graphs.

\begin{proof}
Recall that $\tau(\G) = \sum_{\varphi \subset \G} (-1)^{e(\varphi)}$ where the sum is over all subforests of $\G$. Therefore,
\begin{gather*}
    \sum_{\gamma \subset \G} \tau(\gamma) (-1)^{e(\G/\gamma)}
    =
    (-1)^{e(\G)} \sum_{\gamma \subset \G} \tau(\gamma)  (-1)^{e(\gamma)}
    =
    \\
    (-1)^{e(\G)} \sum_{\gamma \subset \G} \sum_{\substack{ \varphi \subset \gamma\\ \text{ forest } \varphi}} (-1)^{e(\varphi)}(-1)^{e(\gamma)}
    =
    (-1)^{e(\G)} \sum_{\substack{ \varphi \subset \G\\ \text{ forest } \varphi}} (-1)^{e(\varphi)} \sum_{\substack{\gamma \subset \G\\ \gamma \supset \varphi}}  (-1)^{e(\gamma)}.
\end{gather*}
The set of subgraphs of $\G$ containing $\varphi$ is in bijection with the set of subsets of $E(\G) \setminus E(\varphi)$. Because $\G$ has at least one cycle, $E(\G) \setminus E(\varphi)$ is not empty and the alternating sum 
\begin{align*}
\sum_{\substack{\gamma \subset \G\\ \gamma \supset \varphi}}  (-1)^{e(\gamma)} 
&= 
(-1)^{e(\varphi)}\sum_{E' \subset E(\G) \setminus E(\varphi)}  (-1)^{|E'|} = 0. \qedhere
\end{align*}
\end{proof}

\begin{corollary} 
\label{crll:tau_identity_summed}
\begin{align*}
1=  \sum_{\G\in  \GG_0}  \sum_{\gamma\subset\G} \frac{\tau(\gamma)(-1)^{e(\G/\gamma)}}{|\Aut(\G)|}z^{|\G|}.
\end{align*}
\end{corollary}
\begin{proof} Since all non-trivial graphs in $\GG_0$ have cycles,  Proposition~\ref{prop:tau_identity} implies that the only non-zero contribution to the sum comes from the empty graph.
\end{proof}

To eventually obtain the statement of Proposition~\ref{prop:Tzx_graph_counting_identity}, we transform this identity using the following proposition, which is an elementary application of labeled counting.

\begin{proposition}
    \label{prop:convoluted_graph_sum}
Let  $\phi$ be a function from graphs to a formal power series ring such that for each monomial $m$ and each integer $\ell\geq 0$,  the coefficient of $m$ in $\phi(\G)$ is non-zero for only finitely many graphs $\G\in \GG_{2\ell}.$ 
Then
     $$\sum_{\G\in  \GG_0}  \sum_{\gamma\subset\G} \frac{\phi(\gamma)w^{e(\G/\gamma)}}{|\Aut(\G)|}=\sum_{\ell\geq 0} w^\ell (2\ell-1)!!    \sum_{\gamma \in \GG_{2\ell}}   \frac{\phi(\gamma)}{|\Aut \gamma|},$$
     where $w$ is a formal variable.  
\end{proposition}

\begin{proof}  
    To prove the proposition we will use (totally) labeled graphs. 
 Here a {\em labeling} of $\G$ with $e(\G)$ edges, $v(\G)$ vertices and $s(\G)$ leaves consists of 
\begin{itemize}
\item ordering the edges, i.e.\ labeling them     $1,\ldots,e(\G)$,
\item orienting each edge, 
\item ordering the vertices, i.e.\ labeling them  $1,\ldots,v(\G)$,
\item ordering the leaves, i.e.\ labeling them   $1,\ldots,s(\G)$.  
\end{itemize}
The set of labeled graphs with $s$ leaves will be denoted $\LG_s$.

The advantage of using labeled graphs instead of unlabeled graphs is that a sum of terms $1/|\Aut(\G)|$ over unlabeled graphs on $v$ vertices and $e$ edges becomes a sum of $1/(v!e!2^e)$ over labeled graphs using the orbit-stabilizer theorem. The group $\Aut(\G)$ acts on the set of labelings of $\G$, an orbit is a labeled graph and all stabilizers are trivial. This simplifies expressions that involve these automorphism groups.  In particular, abbreviating $v=v(\G), e=e(\G)$ and $d=e(\gamma)$  we have
 \begin{align*}
    \sum_{\G\in  \GG_0}  \sum_{\gamma\subset\G} \frac{\phi(\gamma)w^{e(\G/\gamma)}}{|\Aut(\G)|}
=\sum_{\G\in  \LG_0}  \sum_{\gamma\subset\G}\frac{w^{e-d}}{e!v!2^{e}}\phi(\gamma).
\end{align*}
A subgraph $\gamma$ inherits a labeling from $\G$:  the vertices are the same, so they have the same labels.  The ordering  and orientation on the edges of $\G$  induces an ordering  and orientation on the edges of $\gamma$, giving a labeling on these.  The edges not in $\gamma$ also have an induced ordering, and we use that to order the leaves of $\gamma$ by the following rule:  If there are $\ell$ edges in $E(\G)\setminus E(\gamma)$, label the leaf corresponding to the initial half of the $i$-th  edge by $i$, and the leaf corresponding to the terminal half by $i+\ell$.  

We now change the order of summation.  Remembering that $\gamma$ has an even number of leaves, we get
\begin{align*}
\sum_{\G\in  \LG_0}  \sum_{\gamma\subset\G}\frac{w^{e-d}}{e!v!2^{e}}\phi(\gamma)&=
  \sum_{\ell\geq 0}\sum_{\gamma\in\LG_{2\ell}}   \sum_{\G\in \LG_0, \G\supset \gamma}\frac{w^{e-d}}{e!v!2^{e}}\phi(\gamma)\\
&=\sum_{\ell\geq 0}\sum_{\gamma\in\LG_{2\ell}}   \sum_{\G\in \LG_0, \G\supset \gamma}\frac{w^{\ell}}{(\ell+d)!v!2^{\ell+d}}\phi(\gamma).
  \intertext{
We next note that a labeling on $\gamma$  specifies an isomorphism type of $\Gamma$ containing $\gamma$, using the rule that the $i$-th leaf should be connected to the $(i+\ell)$-th leaf.  This also orders the edges in $E(\Gamma)\setminus E(\gamma)$ and orients them from $i$ to $i+\ell$.  The edges of $\gamma\subset \G$ are ordered and oriented  as subsets of $E(\gamma)$.  There are $\binom{d+\ell}{\ell}$ ways of shuffling the two orderings to get a total ordering on $E(\G)$ that induces  the given orderings on $E(\gamma)$ and $E(\G)\setminus E(\gamma)$.  Thus the last expression becomes 
}
     &=\sum_{\ell\geq 0}\sum_{\gamma\in\LG_{2\ell}} \binom{d+\ell}{\ell} \frac{w^{\ell}}{(d+\ell)!v!2^{d+\ell}}\phi(\gamma) \\
 &=\sum_{\ell\geq 0} \frac{w^\ell}{2^\ell}\sum_{\gamma\in\LG_{2\ell}}  \frac{(d+\ell)!}{\ell!d!}\frac{1}{(d+\ell)!v!2^{d}}\phi(\gamma) \\
 &=\sum_{\ell\geq 0} \frac{1}{\ell!2^\ell} w^\ell \sum_{\gamma\in\LG_{2\ell}}  \frac{\phi(\gamma)}{d!v!2^{d}}.
  \intertext{
We now translate back to unlabeled graphs to get 
}
& =\sum_{\ell \geq 0} \frac{1}{\ell!2^\ell} w^\ell \sum_{\gamma\in\GG_{2\ell}}  \frac{\phi(\gamma)(2\ell)!}{|\Aut(\gamma)|}\\
& =\sum_{\ell \geq 0} \frac{(2\ell)!}{\ell!2^\ell} w^\ell \sum_{\gamma\in\GG_{2\ell}}  \frac{\phi(\gamma)}{|\Aut(\gamma)|}\\
& =\sum_{\ell \geq 0}(2\ell-1)!! w^\ell \sum_{\gamma\in\GG_{2\ell}}  \frac{\phi(\gamma)}{|\Aut(\gamma)|}.
\qedhere
\end{align*}
\end{proof}

\begin{proof}[Proof of Proposition~\ref{prop:Tzx_graph_counting_identity}]
    Use Proposition~\ref{prop:convoluted_graph_sum} with $\phi(\Gamma) = \tau(\Gamma) z^{|\Gamma|}$ and $w = -z$, together with the observation that $(-1)^{e(\G/\gamma)}z^{|\G|}=z^{|\gamma|}(-z)^{e(\G)-e(\gamma)}$. We get
$$
        \sum_{\G \in \GG_0}  \sum_{\gamma \subset \G}  \frac{ \tau(\gamma)(-1)^{e({\G/\gamma})}}{|\Aut \G|}z^{|\G|} 
        =\sum_{\ell\geq 0} (-z)^\ell(2\ell-1)!! \left(\sum_{\gamma\in\GG_{2\ell}} \frac{\tau(\gamma)}{|\Aut(\gamma)|} z^{|\gamma|}\right).
$$
Apply Corollary~\ref{crll:tau_identity_summed} to obtain eq.\ \eqref{eqn:graph_sum_tau_identity} and Corollary~\ref{crll:disconnected_exp} after that.
\end{proof}
\section{The Hopf algebra of core graphs}\label{sec:core}

A graph with no separating edges is called a {\it core graph}, {\it bridgeless} or {\it 1-particle irreducible graph}. Let $\HH$  denote the $\Q$-vector space generated by all finite core graphs. In contrast to the rest of the article, we will  include graphs with bivalent edges as generators of $\HH$. 
The vector space $\HH$ can be made into an algebra whose multiplication is induced by disjoint union of generators; here we identify all graphs with no edges  with the neutral element $\one$ for  this multiplication.  (Thus a topological graph representing the neutral element is a (possibly  empty) disjoint union of isolated vertices and star graphs.)

The algebra $\HH$ can also be equipped with a coproduct $\Delta \colon \HH \to \HH \otimes \HH$, defined by 
\begin{align}
    \label{eqn:coproduct}
\Delta(\G)=\sum_{\gamma\subset\G} \gamma\otimes \G/\gamma,
\end{align}
where the sum is over all core subgraphs of $\G$.

\begin{example}
The graph 
$
{
\def \scale {1.3ex}
\begin{tikzpicture}[x=\scale,y=\scale,baseline={([yshift=-.7ex]current bounding box.center)}] \begin{scope}[node distance=1] \coordinate (v0); \coordinate[right=.5 of v0] (v4); \coordinate[above right= of v4] (v2); \coordinate[below right= of v4] (v3); \coordinate[below right= of v2] (v5); \coordinate[right=.5 of v5] (v1); \coordinate[above right= of v2] (o1); \coordinate[below right= of v2] (o2);  \coordinate[below left=.5 of v0] (i1); \coordinate[above left=.5 of v0] (i2); \coordinate[below right=.5 of v1] (o1); \coordinate[above right=.5 of v1] (o2);  \draw (v0) -- (i1); \draw (v0) -- (i2); \draw (v1) -- (o1); \draw (v1) -- (o2);  \draw (v0) to[bend left=20] (v2); \draw (v0) to[bend right=20] (v3); \draw (v1) to[bend left=20] (v3); \draw (v1) to[bend right=20] (v2);  \draw (v2) to[bend right=60] (v3); \draw (v2) to[bend left=60] (v3);  \filldraw (v0) circle(1pt); \filldraw (v1) circle(1pt); \filldraw (v2) circle(1pt); \filldraw (v3) circle(1pt); \end{scope} \end{tikzpicture}
}
$ has seven mutually non-isomorphic core subgraphs---including the trivial subgraph graph (identified with $\one$) and the graph $
{
\def \scale {1.3ex}
\begin{tikzpicture}[x=\scale,y=\scale,baseline={([yshift=-.7ex]current bounding box.center)}] \begin{scope}[node distance=1] \coordinate (v0); \coordinate[right=.5 of v0] (v4); \coordinate[above right= of v4] (v2); \coordinate[below right= of v4] (v3); \coordinate[below right= of v2] (v5); \coordinate[right=.5 of v5] (v1); \coordinate[above right= of v2] (o1); \coordinate[below right= of v2] (o2);  \coordinate[below left=.5 of v0] (i1); \coordinate[above left=.5 of v0] (i2); \coordinate[below right=.5 of v1] (o1); \coordinate[above right=.5 of v1] (o2);  \draw (v0) -- (i1); \draw (v0) -- (i2); \draw (v1) -- (o1); \draw (v1) -- (o2);  \draw (v0) to[bend left=20] (v2); \draw (v0) to[bend right=20] (v3); \draw (v1) to[bend left=20] (v3); \draw (v1) to[bend right=20] (v2);  \draw (v2) to[bend right=60] (v3); \draw (v2) to[bend left=60] (v3);  \filldraw (v0) circle(1pt); \filldraw (v1) circle(1pt); \filldraw (v2) circle(1pt); \filldraw (v3) circle(1pt); \end{scope} \end{tikzpicture}
}
$ itself.
The coproduct is given by 
{
\def \scale {2ex}
\begin{gather*}
\Delta 
\begin{tikzpicture}[x=\scale,y=\scale,baseline={([yshift=-.7ex]current bounding box.center)}] \begin{scope}[node distance=1] \coordinate (v0); \coordinate[right=.5 of v0] (v4); \coordinate[above right= of v4] (v2); \coordinate[below right= of v4] (v3); \coordinate[below right= of v2] (v5); \coordinate[right=.5 of v5] (v1); \coordinate[above right= of v2] (o1); \coordinate[below right= of v2] (o2);  \coordinate[below left=.5 of v0] (i1); \coordinate[above left=.5 of v0] (i2); \coordinate[below right=.5 of v1] (o1); \coordinate[above right=.5 of v1] (o2);  \draw (v0) -- (i1); \draw (v0) -- (i2); \draw (v1) -- (o1); \draw (v1) -- (o2);  \draw (v0) to[bend left=20] (v2); \draw (v0) to[bend right=20] (v3); \draw (v1) to[bend left=20] (v3); \draw (v1) to[bend right=20] (v2);  \draw (v2) to[bend right=60] (v3); \draw (v2) to[bend left=60] (v3);  \filldraw (v0) circle(1pt); \filldraw (v1) circle(1pt); \filldraw (v2) circle(1pt); \filldraw (v3) circle(1pt); \end{scope} \end{tikzpicture}
=
\begin{tikzpicture}[x=\scale,y=\scale,baseline={([yshift=-.7ex]current bounding box.center)}] \begin{scope}[node distance=1] \coordinate (v0); \coordinate[right=.5 of v0] (v4); \coordinate[above right= of v4] (v2); \coordinate[below right= of v4] (v3); \coordinate[below right= of v2] (v5); \coordinate[right=.5 of v5] (v1); \coordinate[above right= of v2] (o1); \coordinate[below right= of v2] (o2);  \coordinate[below left=.5 of v0] (i1); \coordinate[above left=.5 of v0] (i2); \coordinate[below right=.5 of v1] (o1); \coordinate[above right=.5 of v1] (o2);  \draw (v0) -- (i1); \draw (v0) -- (i2); \draw (v1) -- (o1); \draw (v1) -- (o2);  \draw (v0) to[bend left=20] (v2); \draw (v0) to[bend right=20] (v3); \draw (v1) to[bend left=20] (v3); \draw (v1) to[bend right=20] (v2);  \draw (v2) to[bend right=60] (v3); \draw (v2) to[bend left=60] (v3);  \filldraw (v0) circle(1pt); \filldraw (v1) circle(1pt); \filldraw (v2) circle(1pt); \filldraw (v3) circle(1pt); \end{scope} \end{tikzpicture}
\otimes
\one
+
2~
\begin{tikzpicture}[x=\scale,y=\scale,baseline={([yshift=-.7ex]current bounding box.center)}] \begin{scope}[node distance=1] \coordinate (v0); \coordinate[right=.5 of v0] (v4); \coordinate[above right= of v4] (v2); \coordinate[below right= of v4] (v3); \coordinate[below right= of v2] (v5); \coordinate[right=.5 of v5] (v1); \coordinate[above right= of v2] (o1); \coordinate[below right= of v2] (o2);  \coordinate[below left=.5 of v0] (i1); \coordinate[above left=.5 of v0] (i2); \coordinate[below right=.5 of v1] (o1); \coordinate[above right=.5 of v1] (o2);  \coordinate[above =.5 of v2] (e2); \coordinate[below =.5 of v3] (e3);  \draw (v0) -- (i1); \draw (v0) -- (i2); \draw (v1) -- (o1); \draw (v1) -- (o2); \draw (v2) -- (e2); \draw (v3) -- (e3);  \draw (v0) to[bend left=20] (v2); \draw (v0) to[bend right=20] (v3); \draw (v1) to[bend left=20] (v3); \draw (v1) to[bend right=20] (v2);  \draw (v2) -- (v3);   \filldraw (v0) circle(1pt); \filldraw (v1) circle(1pt); \filldraw (v2) circle(1pt); \filldraw (v3) circle(1pt); \end{scope} \end{tikzpicture}
\otimes 
\begin{tikzpicture}[x=\scale,y=\scale,baseline={([yshift=-.7ex]current bounding box.center)}] \begin{scope}[node distance=1] \coordinate (v) ; \def \rud {.4}; \coordinate [above=.4 of v] (m);  \coordinate (i1) at ([shift=(225:\rud)]v); \coordinate (i2) at ([shift=(255:\rud)]v); \coordinate (i3) at ([shift=(285:\rud)]v); \coordinate (i4) at ([shift=(315:\rud)]v);  \draw (v) -- (i1); \draw (v) -- (i2); \draw (v) -- (i3); \draw (v) -- (i4);  \filldraw (v) circle (1pt);  \draw (m) circle (.4); \end{scope} \end{tikzpicture}
+
2~
\begin{tikzpicture}[x=\scale,y=\scale,baseline={([yshift=-.7ex]current bounding box.center)}] \begin{scope}[node distance=1] \coordinate (v0); \coordinate[right=.5 of v0] (v4); \coordinate[above right= of v4] (v2); \coordinate[below right= of v4] (v3); \coordinate[below right= of v2] (v5); \coordinate[right=.5 of v5] (v1); \coordinate[above right= of v2] (o1); \coordinate[below right= of v2] (o2);  \coordinate[below left=.5 of v0] (i1); \coordinate[above left=.5 of v0] (i2); \coordinate[below right=.5 of v1] (o1); \coordinate[above right=.5 of v1] (o2); \coordinate[above left=.5 of v2] (e2); \coordinate[below left=.5 of v3] (e3);    \draw (v1) -- (o1); \draw (v1) -- (o2);    \draw (v2) -- (e2); \draw (v3) -- (e3); \draw (v1) to[bend left=20] (v3); \draw (v1) to[bend right=20] (v2);  \draw (v2) to[bend right=60] (v3); \draw (v2) to[bend left=60] (v3);   \filldraw (v1) circle(1pt); \filldraw (v2) circle(1pt); \filldraw (v3) circle(1pt); \end{scope} \end{tikzpicture}
\otimes 
\begin{tikzpicture}[x=\scale,y=\scale,baseline={([yshift=-.7ex]current bounding box.center)}] \begin{scope}[node distance=1] \coordinate (v0); \coordinate[right=.4 of v0] (m); \coordinate[right=.4 of m] (v1);  \coordinate[below left=.5 of v0] (i1); \coordinate[above left=.5 of v0] (i2); \coordinate[below right=.5 of v1] (o1); \coordinate[above right=.5 of v1] (o2);  \draw (v0) -- (i1); \draw (v0) -- (i2); \draw (v1) -- (o1); \draw (v1) -- (o2);  \filldraw (v0) circle(1pt); \filldraw (v1) circle(1pt);  \draw (m) circle (.4); \end{scope} \end{tikzpicture}
+
\begin{tikzpicture}[x=\scale,y=\scale,baseline={([yshift=-.7ex]current bounding box.center)}] \begin{scope}[node distance=1] \coordinate (v0); \coordinate[right=.5 of v0] (v4); \coordinate[above right= of v4] (v2); \coordinate[below right= of v4] (v3); \coordinate[below right= of v2] (v5); \coordinate[right=.5 of v5] (v1); \coordinate[above right= of v2] (o1); \coordinate[below right= of v2] (o2);  \coordinate[below left=.5 of v0] (i1); \coordinate[above left=.5 of v0] (i2); \coordinate[below right=.5 of v1] (o1); \coordinate[above right=.5 of v1] (o2);  \coordinate[above left=.5 of v2] (e2); \coordinate[below left=.5 of v3] (e3); \coordinate[above right=.5 of v2] (h2); \coordinate[below right=.5 of v3] (h3);  \draw (v0) -- (i1); \draw (v0) -- (i2); \draw (v1) -- (o1); \draw (v1) -- (o2); \draw (v2) -- (h2); \draw (v3) -- (h3); \draw (v2) -- (e2); \draw (v3) -- (e3);  \draw (v0) to[bend left=20] (v2); \draw (v0) to[bend right=20] (v3); \draw (v1) to[bend left=20] (v3); \draw (v1) to[bend right=20] (v2);     \filldraw (v0) circle(1pt); \filldraw (v1) circle(1pt); \filldraw (v2) circle(1pt); \filldraw (v3) circle(1pt); \end{scope} \end{tikzpicture}
\otimes 
\begin{tikzpicture}[x=\scale,y=\scale,baseline={([yshift=-.7ex]current bounding box.center)}] \begin{scope}[node distance=1] \coordinate (v) ; \def \rad {.8}; \def \rud {.4}; \coordinate (s1) at ([shift=(30:\rad)]v); \coordinate (s2) at ([shift=(150:\rad)]v);  \coordinate (i1) at ([shift=(225:\rud)]v); \coordinate (i2) at ([shift=(255:\rud)]v); \coordinate (i3) at ([shift=(285:\rud)]v); \coordinate (i4) at ([shift=(315:\rud)]v);  \draw (v) -- (i1); \draw (v) -- (i2); \draw (v) -- (i3); \draw (v) -- (i4);  \draw (v) to[out=90,in=120] (s1) to[out=-60,in=0-30] (v); \draw (v) to[out=210,in=240] (s2) to[out=60,in=90] (v); \filldraw (v) circle (1pt); \end{scope} \end{tikzpicture}
\\
+
4~
\begin{tikzpicture}[x=\scale,y=\scale,baseline={([yshift=-.7ex]current bounding box.center)}] \begin{scope}[node distance=1] \coordinate (v0); \coordinate[right=.5 of v0] (v4); \coordinate[above right= of v4] (v2); \coordinate[below right= of v4] (v3); \coordinate[below right= of v2] (v5); \coordinate[right=.5 of v5] (v1); \coordinate[above right= of v2] (o1); \coordinate[below right= of v2] (o2);  \coordinate[below left=.5 of v0] (i1); \coordinate[above left=.5 of v0] (i2); \coordinate[below right=.5 of v1] (o1); \coordinate[above right=.5 of v1] (o2); \coordinate[above left=.5 of v2] (e2); \coordinate[below left=.5 of v3] (e3); \coordinate[above right=.5 of v2] (h2); \coordinate[below right=.5 of v3] (h3);    \draw (v1) -- (o1); \draw (v1) -- (o2);    \draw (v2) -- (e2); \draw (v3) -- (e3); \draw (v2) -- (h2); \draw (v3) -- (h3); \draw (v1) to[bend left=20] (v3); \draw (v1) to[bend right=20] (v2);   \draw (v2) -- (v3);   \filldraw (v1) circle(1pt); \filldraw (v2) circle(1pt); \filldraw (v3) circle(1pt); \end{scope} \end{tikzpicture}
\otimes 
\begin{tikzpicture}[x=\scale,y=\scale,baseline={([yshift=-.7ex]current bounding box.center)}] \begin{scope}[node distance=1] \coordinate (v0); \coordinate[right=.4 of v0] (m); \coordinate[right=.4 of m] (v1); \coordinate[right=.3 of v1] (v2);  \coordinate[below left=.5 of v0] (i1); \coordinate[above left=.5 of v0] (i2); \coordinate[below =.5 of v1] (o1); \coordinate[above =.5 of v1] (o2);  \draw (v0) -- (i1); \draw (v0) -- (i2); \draw (v1) -- (o1); \draw (v1) -- (o2);   \filldraw (v0) circle(1pt); \filldraw (v1) circle(1pt);  \draw (v2) circle (.3); \draw (m) circle (.4); \end{scope} \end{tikzpicture}
+
\begin{tikzpicture}[x=\scale,y=\scale,baseline={([yshift=-.7ex]current bounding box.center)}] \begin{scope}[node distance=1] \coordinate (v0); \coordinate[right=.5 of v0] (v4); \coordinate[above right= of v4] (v2); \coordinate[below right= of v4] (v3); \coordinate[below right= of v2] (v5); \coordinate[right=.5 of v5] (v1); \coordinate[above right= of v2] (o1); \coordinate[below right= of v2] (o2);  \coordinate[below left=.5 of v0] (i1); \coordinate[above left=.5 of v0] (i2); \coordinate[below right=.5 of v1] (o1); \coordinate[above right=.5 of v1] (o2);  \coordinate[above left=.5 of v2] (e2); \coordinate[below left=.5 of v3] (e3); \coordinate[above right=.5 of v2] (h2); \coordinate[below right=.5 of v3] (h3);      \draw (v2) -- (h2); \draw (v3) -- (h3); \draw (v2) -- (e2); \draw (v3) -- (e3);       \draw (v2) to[bend right=60] (v3); \draw (v2) to[bend left=60] (v3);    \filldraw (v2) circle(1pt); \filldraw (v3) circle(1pt); \end{scope} \end{tikzpicture}
\otimes 
\begin{tikzpicture}[x=\scale,y=\scale,baseline={([yshift=-.7ex]current bounding box.center)}] \begin{scope}[node distance=1] \coordinate (v0); \coordinate[right=.4 of v0] (m1); \coordinate[right=.4 of m1] (v1); \coordinate[right=.4 of v1] (m2); \coordinate[right=.4 of m2] (v2);  \coordinate[below left=.5 of v0] (i1); \coordinate[above left=.5 of v0] (i2); \coordinate[below right=.5 of v2] (o1); \coordinate[above right=.5 of v2] (o2);  \draw (v0) -- (i1); \draw (v0) -- (i2); \draw (v2) -- (o1); \draw (v2) -- (o2);  \filldraw (v0) circle(1pt); \filldraw (v1) circle(1pt); \filldraw (v2) circle(1pt);  \draw (m1) circle (.4); \draw (m2) circle (.4); \end{scope} \end{tikzpicture}
+
\one 
\otimes 
\begin{tikzpicture}[x=\scale,y=\scale,baseline={([yshift=-.7ex]current bounding box.center)}] \begin{scope}[node distance=1] \coordinate (v0); \coordinate[right=.5 of v0] (v4); \coordinate[above right= of v4] (v2); \coordinate[below right= of v4] (v3); \coordinate[below right= of v2] (v5); \coordinate[right=.5 of v5] (v1); \coordinate[above right= of v2] (o1); \coordinate[below right= of v2] (o2);  \coordinate[below left=.5 of v0] (i1); \coordinate[above left=.5 of v0] (i2); \coordinate[below right=.5 of v1] (o1); \coordinate[above right=.5 of v1] (o2);  \draw (v0) -- (i1); \draw (v0) -- (i2); \draw (v1) -- (o1); \draw (v1) -- (o2);  \draw (v0) to[bend left=20] (v2); \draw (v0) to[bend right=20] (v3); \draw (v1) to[bend left=20] (v3); \draw (v1) to[bend right=20] (v2);  \draw (v2) to[bend right=60] (v3); \draw (v2) to[bend left=60] (v3);  \filldraw (v0) circle(1pt); \filldraw (v1) circle(1pt); \filldraw (v2) circle(1pt); \filldraw (v3) circle(1pt); \end{scope} \end{tikzpicture}.
\end{gather*}%
}
Note that the complete contraction $\G/\G$ has no edges, so is identified with $\one$. Also note that we omitted isolated vertices of the subgraphs on the left hand side of the tensor product since isolated vertices are also identified with $\one$.
\end{example}

This coproduct is coassociative, making $\HH$ into a bialgebra.  In fact  Kreimer showed that $\HH$ has a  Hopf algebra structure \cite{kreimer2009core}.  The unit $\unit \colon \Q \to \HH$ sends $\unit \colon q \mapsto q \one$.  The co-unit $\counit \colon \HH\to \Q$ sends $\one$ to $1\in \Q$ and all  other graphs to $0$.
The antipode $S\colon \HH\to \HH$ can be defined inductively by $S(\one)=\one$ and, 
\begin{align*} 
 S(\G)=-\sum_{\gamma\subsetneq \G} S(\gamma)\G/\gamma \text{ for } \Gamma \neq \one,
\end{align*}
where the sum is over all core subgraphs of $\G$ which are not equal to $\G$. This recursion terminates since the graphs $\gamma$ in the sum have fewer independent cycles (i.e. smaller first Betti number)  than $\G$. The result is a polynomial in core graphs. 
We refer the reader to \cite{sweedler1969hopf} for a general account of Hopf algebras or \cite[Ch.\ 3]{borinsky2018graphs} for  more information about this specific Hopf algebra.
 
 A {\em character} on $\HH$ is a linear map $\phi$ which satisfies $\phi(\G_1 \G_2) = \phi(\G_1) \phi(\G_2).$ The convolution $\phi \star \psi$ of two characters 
is defined by $$(\phi\star \psi)(\Gamma)=\sum_{\gamma \subset \G} \phi(\gamma) \psi(\Gamma/\gamma),$$ where we again sum over all core subgraphs of $\G$. Because $\HH$ is a Hopf algebra, the set of all characters from $\HH$ to any commutative algebra forms a group under the convolution product. This follows from the antipode being the inverse to the identity map, $\id \colon \HH \to \HH$, in the sense that $\id \star S = S \star \id= \unit \circ \counit$. The map $\unit \circ \counit \colon \HH \to \HH$ is the identity element of the $\star$-group of characters $\HH \to \HH$. It satisfies $\unit \circ \counit(\one) = \one$ and $\unit \circ \counit(\G) = 0$ if $\G \neq \one$. If $\phi$ is a character $\HH \to \mathcal{A}$ which maps to a unital commutative algebra $\mathcal{A}$, then $\phi^{\star -1} := \phi \circ S$ is the inverse of $\phi$ under the star product in the sense that $$ \phi^{\star -1} \star \phi = \phi \star \phi^{\star -1} = \unit_{\mathcal{A}} \circ \counit,$$ where $\unit_{\mathcal{A}}$ is the unit of $\mathcal{A}$.

Because $\tau$ is multiplicative on disjoint unions of graphs, it induces a character $\HH \to \Q$. We can define the even simpler character $\sigma(\Gamma) = (-1)^{e(\Gamma)}$ and formulate Proposition~\ref{prop:tau_identity} in the Hopf algebra language:
\begin{proposition}
$\tau \star \sigma = \sigma \star \tau = \unit_\Q \circ \counit$. 
\end{proposition}

\begin{proof}
By Proposition~\ref{prop:tau_identity} and the definition of the $\star$ product  $\tau \star \sigma = \unit_\Q \circ \counit$. Because the characters form a group, we also have $\sigma \star \tau = \unit_\Q \circ \counit$.
\end{proof}

Although the Hopf algebra $\HH$ and its coproduct are  defined only on core graphs, we can also consider the maps $\tau$ and $\sigma$ on the space of all graphs.
 The linear space $\HG$ which is generated by all graphs can be made into a (left) $\HH$-comodule by defining a \textit{coaction}, $\rho \colon \HG \to \HH \otimes \HG$,  using the formula \eqref{eqn:coproduct} with $\rho$ in place of $\Delta$. The left side of the tensor product in \eqref{eqn:coproduct} will always be a core graph and can naturally be associated with an element in $\HH$. The star product applied on characters of $\HG$ in the same way as on characters of $\HH$ becomes an \textit{action} this way. See \cite[Ch.\ 3]{borinsky2018graphs} for details. 

Applying $\sigma$ to the weighted sum of all connected graphs with no leaves gives an especially interesting result: 
\begin{proposition}
\label{prop:bernoulli_graphs_sum}
\begin{align*}
    \sum_{\substack{\Gamma\in \GG^c_0\\|\G|=n}}\frac{\sigma(\G)}{|\Aut(\G)|}=\frac{\zeta(-n)}{n}=-\frac{B_{n+1}}{n(n+1)} \text{ for all } n \geq 1.
\end{align*}
\end{proposition}
This statement is not new. It follows as a special case from `Penner's model' \cite{penner1986moduli} (see also \cite[Appendix\ D]{Kon92}). 
The sum could be thought of as the integral of $\sigma$ over the space of connected graphs with measure $\mu(\G)=1/|\Aut(\G)|$, whereas integrating its convolutive inverse $\tau(\G)$ over the same space with the same measure gives $\ch_n$ by Corollary~\ref{cor:SV}. 

In Section~\ref{sec:laplace} we will give a proof of Proposition~\ref{prop:bernoulli_graphs_sum} as a special case of Corollary~\ref{crll:graph_laplace}. Here, we can immediately verify it for $n=1$:
\begin{align*}
\frac{\sigma(
{
\begin{tikzpicture}[x=1ex,y=1ex,baseline={([yshift=-.6ex]current bounding box.center)}] \coordinate (vm); \coordinate [left=.7 of vm] (v0); \coordinate [right=.7 of vm] (v1); \draw (v0) circle(.7); \draw (v1) circle(.7); \filldraw (vm) circle (1pt); \end{tikzpicture}%
}
)}{|\Aut(
{
\begin{tikzpicture}[x=1ex,y=1ex,baseline={([yshift=-.6ex]current bounding box.center)}] \coordinate (vm); \coordinate [left=.7 of vm] (v0); \coordinate [right=.7 of vm] (v1); \draw (v0) circle(.7); \draw (v1) circle(.7); \filldraw (vm) circle (1pt); \end{tikzpicture}%
}
)|}
+
\frac{\sigma(
{
\begin{tikzpicture}[x=1ex,y=1ex,baseline={([yshift=-.6ex]current bounding box.center)}] \coordinate (v0); \coordinate [right=1.5 of v0] (v1); \coordinate [left=.7 of v0] (i0); \coordinate [right=.7 of v1] (o0); \draw (v0) -- (v1); \filldraw (v0) circle (1pt); \filldraw (v1) circle (1pt); \draw (i0) circle(.7); \draw (o0) circle(.7); \end{tikzpicture}%
}
)}{|\Aut(
{
\begin{tikzpicture}[x=1ex,y=1ex,baseline={([yshift=-.6ex]current bounding box.center)}] \coordinate (v0); \coordinate [right=1.5 of v0] (v1); \coordinate [left=.7 of v0] (i0); \coordinate [right=.7 of v1] (o0); \draw (v0) -- (v1); \filldraw (v0) circle (1pt); \filldraw (v1) circle (1pt); \draw (i0) circle(.7); \draw (o0) circle(.7); \end{tikzpicture}%
}
)|}
+
\frac{\sigma(
{
\begin{tikzpicture}[x=1ex,y=1ex,baseline={([yshift=-.6ex]current bounding box.center)}] \coordinate (vm); \coordinate [left=1 of vm] (v0); \coordinate [right=1 of vm] (v1); \draw (v0) -- (v1); \draw (vm) circle(1); \filldraw (v0) circle (1pt); \filldraw (v1) circle (1pt); \end{tikzpicture}%
}
)}{|\Aut(
{
\begin{tikzpicture}[x=1ex,y=1ex,baseline={([yshift=-.6ex]current bounding box.center)}] \coordinate (vm); \coordinate [left=1 of vm] (v0); \coordinate [right=1 of vm] (v1); \draw (v0) -- (v1); \draw (vm) circle(1); \filldraw (v0) circle (1pt); \filldraw (v1) circle (1pt); \end{tikzpicture}%
}
)|}
&=
\frac{1}{8} + \frac{-1}{8} + \frac{-1}{12} = 
-\frac{1}{12} = -\frac{B_2}{2}.
\end{align*}

The Bernoulli numbers are classical objects with a long history, and it is well-known that $B_{2n+1}$ vanishes for $n \geq 1$ and that the sign of $B_{2n}$ is $(-1)^{n+1}$ for $n\geq 1$. To analyse similar properties of the numbers $\ch_n$, we will make heavy use of \textit{asymptotic expansions}.  We will go into the details after a short digression about the relation of our methods with perturbative methods used in quantum field theory.

\section{Renormalized topological quantum field theory}\label{sec:tqft}

Our approach to analyzing the numbers $\ch_n$ is in line with an established technique for analyzing topological objects by using \textit{perturbative quantum field theory} or equivalently \textit{Feynman diagram techniques} \cite{bessis1980quantum}. The term \textit{topological quantum field theory} is used for a quantum field theory whose observables are topological invariants \cite{witten1988topological}. See also \cite{kontsevich1993formal,kontsevich1994feynman} for further aspects of this theory and \cite{conant2003theorem} for a more detailed account focused on group cohomology. 

One prominent application of topological quantum field theory is intersection theory in the moduli space of complex curves, as developed by Witten \cite{witten1990two} and Kontsevich \cite{Kon92}.  Penner \cite{penner1986moduli} had already applied perturbative quantum field theory techniques to reprove the result of Harer and Zagier %
on the rational Euler characteristic of the mapping class group. In the course of his study of intersection theory Kontsevich gave a simplified version of Penner's proof \cite[Appendix\ D]{Kon92}.  This simplified proof involves a formula similar to  the one in Proposition~\ref{prop:bernoulli_graphs_sum}. 

We can endow our approach to studying the numbers $\ch_n$ with a quantum field theoretical interpretation, in  a spirit similar to the work of   Penner and Kontsevich.  Here is   a brief, heuristic indication of how this goes.

We start with the statement of Proposition~\ref{prop:Tzx_graph_counting_identity} and immediately apply Proposition~\ref{prop:Tzx_leaves_identity} to obtain the equation
\begin{align*}
    1 &= \sum_{\ell \geq 0} (-z)^\ell (2\ell-1)!! [x^{2\ell}] \exp\left( \frac{e^x-\frac{x^2}{2}-x-1}{z} + \frac{x}{2} + T(ze^{-x}) \right).
\end{align*}
Now flip the sign of $z$ to get
\begin{align}
\label{eqn:Tzx_graph_counting_identity_qft_expl}
    1 &= \sum_{\ell \geq 0} z^\ell (2\ell-1)!! [x^{2\ell}] \exp\left( -\frac{e^x-\frac{x^2}{2}-x-1}{z} + \frac{x}{2} + T(-ze^{-x}) \right).
\end{align}
For the remainder of this section we regard $z$   not as a formal variable, but rather as a positive real number. We then recall the Gaussian integrals 
\begin{align*}
 \frac{1}{\sqrt{2\pi z}} \int_\R x^{2 \ell} e^{-\frac{x^2}{2 z}} dx &= z^\ell (2\ell-1)!! \text{ for all } \ell \geq 0 \\
 \frac{1}{\sqrt{2\pi z}} \int_\R x^{2 \ell+1} e^{-\frac{x^2}{2 z}} dx &=0 \text{ for all } \ell \geq 0.
\end{align*}
Substituting these into eq.\ \eqref{eqn:Tzx_graph_counting_identity_qft_expl} gives
\begin{align*}
    1 &= \sum_{\ell \geq 0} \frac{1}{\sqrt{2\pi z}} \int_\R x^{2 \ell} e^{-\frac{x^2}{2 z}} dx [x^{2\ell}] \exp\left( -\frac{e^x-\frac{x^2}{2}-x-1}{z} + \frac{x}{2} + T(-ze^{-x}) \right).
\end{align*}
This integral is not convergent since we are no longer regarding $z$ as a formal variable, but we will disregard this issue for this heuristic argument. In the same laissez-faire spirit, we ignore convergence issues and interchange summation with integration to obtain
\begin{align}
    \label{eqn:renormalization_condition}
    1 &= \frac{1}{\sqrt{2\pi z}} \int_\R \exp\left( -\frac{e^x-x-1}{z} + \frac{x}{2} + T(-ze^{-x}) \right) dx.
\end{align}
This integral is again not well-defined, as the series $T(z)$ does not converge to a function of  $z$ in any finite domain: it is only a formal power series with a vanishing radius of convergence. However, we can interpret the right hand side of this equation as a `path-integral' of a \textit{zero-dimensional quantum field theory} with the \textit{action} $-(e^x-x-1)$, where the parameter $z$ takes the role of \textit{Planck's constant} $\hbar$. The additional terms in the exponent $\frac{x}{2} + T(-ze^{-x})$ can be interpreted as \textit{counterterms} or \textit{renormalization constants} which \textit{renormalize} the quantum field theory in a generalized sense. In fact, equation~\eqref{eqn:renormalization_condition} can be interpreted as a \textit{renormalization condition} of a quantum field theory. 

In Kontsevich's proof of the Harer-Zagier formula, a topological quantum field theory was constructed whose perturbative expansion encoded the geometric invariants of interest. As we have seen above, our method can also be interpreted as an application of quantum field theory to the analysis of  the invariants $\ch_n$. However, instead of using the coefficients of the perturbative expansion directly, we use the coefficients of the renormalization constants to express the quantities which are of interest. We might therefore say that we are using a \textit{renormalized topological quantum field theory} to encode $\ch_n$. 

This is consistent with the interpretation of $\tau$ as a character on the core Hopf algebra. Connes and Kreimer \cite{connes2000renormalization} showed that the renormalization procedure in quantum field theory can be seen as the solution of a Riemann-Hilbert problem using a Birkhoff decomposition. The Birkhoff decomposition can be formulated elegantly as an inversion in the group of characters of a certain Hopf algebra. In our topological case, which is much simpler than the full physical picture, this interpretation boils down to the brief exposition in Section~\ref{sec:core}. Consult \cite{borinsky2017renormalized} for a general treatment of renormalized zero-dimensional quantum field theory in a Hopf algebraic framework. 

After these expository remarks we  now return to our rigorous treatment of the Euler characteristic of $\Outn$.

\newcommand{\llrrparen}[1]{%
  \left(\mkern-3mu\left(#1\right)\mkern-3mu\right)}
\section{Asymptotic expansions}
\label{sec:asymptotic_expansions}

An often useful approach to studying a generating function such as $T(z)= \sum_{n\geq 1} \ch_n z^n$ is to interpret it as an analytic function in $z$ and then use analytic techniques to study the nature of its coefficients \cite{flajolet2009analytic}. However, in our case this standard approach is doomed to fail, at least if it is applied naively, as the coefficients of $T(z)$  turn out to grow factorially so the power series $T(z)$ has a vanishing radius of convergence.

We will circumvent this problem by using an \textit{asymptotic expansion} of a certain function to describe the coefficients of $T(z)$. In contrast to Taylor expansions of analytic functions, asymptotic expansions are not necessarily convergent in any non-vanishing domain of $\C$.

\subsection{Asymptotic notation}

In this section we fix the notation we use for asymptotic expansions  and prove a basic property that we will use repeatedly.   
We begin by recalling the  {\em big  $\bigO$} and {\em small $\smallO$}   notation.
Let $f, g$ and $h$ be  functions defined on a domain $D$ and let $L$ be a limit point of $D$. 
The notation $f(x)=g(x) + \bigO(h(x))$ means $f-g\in \bigO(h)$, where $\bigO(h)$ is the set of all functions $u$ defined on $D$ such that 
\begin{align*}
    \limsup_{x\to L} \left| \frac{u(x)}{h(x)} \right| <  \infty.
\end{align*}
Similarly, $f(x)=g(x)+\smallO(h(x))$ means $f-g\in \smallO(h),$  where $\smallO(h)$ consists of  all functions $u$ that satisfy $\lim_{x\to L} \frac{u(x)}{h(x)} = 0$. %

An {\it asymptotic scale} on $D$ with respect to a limit $L$ is a sequence of functions $\{\varphi_k\}_{k\geq 0}$ with the property $\varphi_{k+1} \in \smallO(\varphi_k)$ for $k \geq 0$. A common example, for functions with domain $\R$ and limit  $L=\infty$, is $\varphi_k(x)=x^{-k}$.

\begin{definition}
    \label{def:asymptotic_expansion}
    An {\it asymptotic expansion} of a function $f$ defined on $D$ with respect to the limit $L$ and the asymptotic scale $\{\varphi_k\}_{k\geq0}$ is a sequence of coefficients $c_k$ such that
\begin{align*}
    f(x) = \sum_{k=0}^{R-1} c_k\varphi_k(x) + \bigO(\varphi_R(x)) \text{ for all } R \geq 0,
\end{align*}
where the $\bigO$ refers to the limit $x \rightarrow L$. %
We will write this infinite set of $\bigO$ relations as, 
\begin{align*}
    f(x)\sim \sum_{k\geq 0} c_k\varphi_k(x) \text{ as } x \rightarrow L.
\end{align*}
\end{definition}
Asymptotic expansions are widely used in mathematical analysis, the physical sciences and engineering to obtain very accurate approximations to functions.  
A detailed introduction to asymptotic expansions can be found in de Bruijn's book \cite{de1981asymptotic}. A key feature of asymptotic expansions is that, for a given function $f$, limit $L$ and asymptotic scale $\{\varphi_k\}_{k \geq 0}$, the coefficients $c_k$ are unique if they exist. We will make use of this property in the proof of Theorem~\ref{thm:asymptotic_expansion}.

The coefficients of an asymptotic expansion depend on the choice of the asymptotic scale. However, under certain conditions we can translate between asymptotic expansions in different asymptotic scales:
\begin{lemma}
    \label{lmm:scale_change}
    Suppose  $\Phi=\{\varphi_k\}_{k\geq0}$ and $\Psi=\{\psi_m\}_{m\geq0}$ are two  asymptotic scales on a domain $D$ with  respect to the same limit $L$, and suppose  $f$ has an asymptotic expansion in   $\Phi$
\begin{equation}
    \label{eqn:scale_change_original}
    f(x) \sim \sum_{k \geq 0} c_k \varphi_k(x) \text{ as } x\rightarrow L.
\end{equation}
If each $\psi_m$ also has an asymptotic expansion in   $\Phi$
\begin{align}
    \label{eqn:scale_change_scale_relation}
\psi_m(x)\sim \sum_{k\geq m} c_{m,k} \varphi_k(x) \text{ as } x\rightarrow L
\end{align}
with $c_{m,m}\neq 0$, then $f$ has an asymptotic expansion in  $\Psi$
\begin{align*}
   f(x) \sim \sum_{m \geq 0} c'_m \psi_m(x) \text{ as } x\rightarrow L,
\end{align*}
 where the coefficients $c_m'$ are implicitly determined by the infinite triangular equation system $c_k = \sum_{m= 0}^{k} c_m' c_{m,k}$ for all $k\geq 0$.
\end{lemma}

\begin{proof}
By the definition of an asymptotic expansion we have
 \begin{align*}
        \psi_m- \sum^{R-1}_{k=m} c_{m,k} \varphi_k\in  \bigO(\varphi_R) \text{ for all } R \geq m \geq 0.
    \end{align*}
We can multiply a function in $\bigO(h)$  by a constant or add a finite number of functions in $\bigO(h)$ without changing the $\bigO$ class. Thus multiplying by $c_m'$ and then adding from $m=0$ to $R-1$ gives
 \begin{align*}
        \sum_{m=0}^{R-1} c'_m \psi_m - \sum_{m=0}^{R-1}\sum^{R-1}_{k=m} c'_mc_{m,k} \varphi_k \in  \bigO(\varphi_R) \text{ for all } R   \geq 0.
 \end{align*}
 Changing the order of summation and using the definition of the constants $c'_m$ gives 
  \begin{align*}
        \sum_{m=0}^{R-1} c'_m \psi_m- \sum_{k=0}^{R-1}\sum_{m=0}^{k} c'_mc_{m,k} \varphi_k =  \sum_{m=0}^{R-1} c'_m \psi_m - \sum_{k=0}^{R-1}c_k \varphi_k \in  \bigO(\varphi_R) \text{ for all } R   \geq 0.
 \end{align*}
By eq.\ \eqref{eqn:scale_change_original} we have  $f- \sum_{k=0}^{R-1} c_k\varphi_k \in \bigO(\varphi_R)$, so  combining this with the above gives $$f- \sum_{m=0}^{R-1} c_m'\psi_m \in \bigO(\varphi_R) \text{ for all } R\geq 0.$$
It remains only to check that $\bigO(\varphi_R)=\bigO(\psi_R)$.  This follows from eq.\ \eqref{eqn:scale_change_scale_relation}, which   implies $\psi_R=c_{R,R}\varphi_R +  \bigO(\varphi_{R+1}),$ together with the assumption that $c_{R,R}\neq 0$ and the fact that $\varphi_{R+1}\in \smallO(\varphi_R)$. 
 \end{proof}

 In this paper the domain of our functions will mostly be the natural numbers, i.e.\ our functions are sequences $f\colon\N \rightarrow \R$, and the limit will almost always be $\infty$, but the asymptotic scale will vary.   

\subsection{Stirling's approximation}
Arguably, one of the most studied asymptotic expansions is \textit{Stirling's approximation}. This is an asymptotic expansion of the  gamma  function 
\begin{align}
    \label{eqn:stirling_approximation}
    \Gamma(n) \sim \sum_{k \geq 0} \widehat b_k \sqrt{2\pi} e^{-n}n^{n-\frac12-k} \text{ as } n \rightarrow \infty,
\end{align}
where $\widehat b_k$ is the   coefficient of $z^k$ in $\exp\left( \sum_{k=1}^\infty \frac{B_{k+1}}{k(k+1)} z^k \right)$. See for instance \cite[Sec.\ 3.10]{de1981asymptotic} for a proof.
Stirling's approximation is used extensively as a tool for approximating the value of $\Gamma(n)$ for large $n$. %
We, however, will view eq.\ \eqref{eqn:stirling_approximation} as an asymptotic expansion of  $\Gamma(n)$ in the asymptotic scale $\{\sqrt{2 \pi} e^{-n} n^{n-\frac12-k}\}_{k\geq0}$ and use it merely as a tool to encode and manipulate the coefficients $\widehat b_k$.  

Recall that  the gamma function satisfies $\Gamma(z+1) = z\Gamma(z);$ this ensures that the sequence of functions $\{\Gamma( n -k + \frac12 )\}_{k\geq0}$ forms an asymptotic scale in the limit $n \rightarrow \infty$. The statement of Theorem~\ref{thm:asymptotic_expansion} gives an asymptotic expansion of $f(n)= \sqrt{2 \pi} e^{-n} n^{n}$ in this scale, whose    coefficients coincide with those of the formal power series $\exp(T(z))$; we can think of this as a kind of ``inverted'' Stirling's approximation. 
 
Although there is a large and growing literature on Stirling's approximation (see \cite{borwein2018gamma} for a recent survey), such an asymptotic expansion of $\sqrt{2 \pi} e^{-n} n^{n}$ does not seem to have been studied previously. 
This type of `inverted Stirling's approximation' might also be relevant for other applications: many problems dictate or suggest an inherent asymptotic scale. For instance, it might be natural to work in the asymptotic scale $\{(2(n-k)-1)!!\}_{k\geq 0}$, where $(2(n-k)-1)!! = 2^{n-k} \Gamma(n-k+\frac12)/\sqrt{\pi}$, for counting problems whose solution involves double factorials. Moreover, power series with coefficients which have an asymptotic expansion in the scale $\{\Gamma(n-k+\beta)\}_{k\geq 0}$ with $\beta \in \R$ have a rich algebraic structure; for instance they are closed under multiplication and functional composition \cite{borinsky2018generating}.

To establish the asymptotic expansion in Theorem~\ref{thm:asymptotic_expansion} we will start with a trivial asymptotic expansion for the constant function $1$ in the scale $\{n^{-k}\}_{k\geq0}$, then use Lemma~\ref{lmm:scale_change}  to change to the scale $\{\psi_m\}_{m\geq0}$, where 
\begin{align} \label{eqn:psi_defn}
\psi_m(n) = \frac{\Gamma( n -m +\frac12 )}{\sqrt{2 \pi} e^{-n} n^{n}}.
\end{align}
In order to apply the lemma, we need to find asymptotic expansions for the functions $
\psi_m(n)$. We do this using the following variant of Stirling's approximation. 
\begin{proposition}
    \label{prop:graph_stirling}
    Let $\Psi=\{\psi_m\}_{m\geq0}$ be the asymptotic scale with domain $\N$ and limit $\infty$ defined in eq.\ \eqref{eqn:psi_defn}.
    Then each $\psi_m$ has an asymptotic expansion in the asymptotic scale $\{n^{-k}\}_{k\geq0}$ given by
\begin{align*}
\psi_m(n)
\sim \sum_{k \geq m} c_{m,k} n^{-k}
\text{ as } n \rightarrow \infty,
\end{align*}
where $c_{m,k}$ is the coefficient of $z^k$ in the formal power series 
 \begin{align}
    \label{eqn:graph_stirling_asymp_psi}
z^m
\sum_{\ell \geq 0}
z^{\ell} (2\ell-1)!!
[x^{2\ell}]
e^{-\frac{1}{z}\left(e^x - \frac{x^2}{2} - x - 1\right) + x \left( \frac12 - m \right)}.
\end{align}
\end{proposition}

We will prove this proposition using \textit{Laplace's method}, which serves as a connection between graphical enumeration and asymptotic expansions. We will introduce this method in the next section and therefore postpone the proof of Proposition~\ref{prop:graph_stirling} until then. 

Assuming Proposition~\ref{prop:graph_stirling} we are now ready to  prove Theorem~\ref{thm:asymptotic_expansion}.

\begin{repthmx}{thm:asymptotic_expansion}
    The function $\sqrt{2 \pi}e^{-n} n^n$ has the following asymp\-totic expansion in the asymptotic scale $\{(-1)^k  \Gamma( n + \frac12 - k ) \}_{k\geq0}$,
\begin{align*}
    \sqrt{2 \pi}e^{-n} n^n &\sim \sum_{ k\geq 0 } \Ch_k (-1)^k  \Gamma\left( n + \frac12 - k \right) \text{ as } n\rightarrow \infty,
\end{align*}
where $\Ch_k$ is the coefficient of $z^k$ in the formal power series $\exp\left( \sum_{n\geq 1} \chi( \Out (F_{n+1}) ) z^n \right)$.
\end{repthmx}

\begin{proof}
    The constant function $f(n)\equiv 1$ has a trivial asymp\-totic expansion in the asymptotic scale $\{n^{-k}\}_{k\geq0}$, namely
    \begin{align*}
        1 &\sim \sum_{k \geq 0} c_k n^{-k} \text{ as } n \rightarrow \infty,
    \end{align*}
    with   coefficients $c_0 = 1$ and $c_k=0$ for all $k\geq 1$. Using Lemma~\ref{lmm:scale_change} and Proposition~\ref{prop:graph_stirling} we can change the asymptotic scale from $\{n^{-k}\}_{k\geq0}$ to the scale $\Psi$ as defined in Proposition~\ref{prop:graph_stirling}, giving
    \begin{align*}
        1 &\sim \sum_{m \geq 0} c_m' \psi_m(n) \text{ as } n \rightarrow \infty,
    \end{align*}
where the coefficients $c_m'$ are uniquely determined by the triangular equation system 
\begin{align}
\label{eqn:triangular_system} 
c_k = \sum_{m = 0}^{k} c_m' c_{m,k} \text{ for all } k \geq 0
\end{align}
and the coefficients $c_{m,k}$ are those defined  in  the statement of Proposition~\ref{prop:graph_stirling}. Namely, 
$c_{m,k}$ is the coefficient of $z^k$ in the formal power series given in eq.\ \eqref{eqn:graph_stirling_asymp_psi}.  It follows from this power series representation that $c_{m,m} \neq 0$ for all $m \geq 0$, which justifies our application of Lemma~\ref{lmm:scale_change} and guarantees that the linear equation system~\eqref{eqn:triangular_system} can be uniquely solved for the coefficients $c_m'$. 

By definition of $\psi_m$ in eq.\ \eqref{eqn:psi_defn}, this asymptotic expansion becomes
 \begin{align*}
        1 &\sim \sum_{m \geq 0}  c'_m\frac{\Gamma( n -m +\frac12 )}{\sqrt{2 \pi} e^{-n} n^{n}} \text{ as } n \rightarrow \infty.
    \end{align*}
  Multiplying both sides by $\sqrt{2 \pi} e^{-n} n^{n}$ gives 
  \begin{align*}
        \sqrt{2 \pi} e^{-n} n^{n} &\sim \sum_{m \geq 0}  c'_m\Gamma( n -m +\frac12 ) \text{ as } n \rightarrow \infty.
    \end{align*}
 It remains to show that $c_m'=(-1)^m\Ch_m$. 
From Proposition~\ref{prop:Tzx_graph_counting_identity} 
and Proposition~\ref{prop:Tzx_leaves_identity} we have
    \begin{align*}
1 &= \sum_{\ell \geq 0} (-z)^\ell (2\ell-1)!! [x^{2\ell}] \exp\left( T(z,x) \right)\\
&= 
\sum_{\ell \geq 0} (-z)^\ell (2\ell-1)!! [x^{2\ell}] \exp\left( \frac{e^x-\frac{x^2}{2}-x-1}{z} + \frac{x}{2} + T(ze^{-x}) \right)\\
          &=\sum_{\ell \geq 0} z^{\ell} (2\ell-1)!! [x^{2\ell}] e^{-\frac{1}{z} \left( e^x-\frac{x^2}{2} - x - 1 \right) + \frac12 x} \exp\left( T(-ze^{-x}) \right).
    \end{align*}
 Expanding the second exponential in $z$ gives,
    \begin{align*}
        1 &= \sum_{\ell \geq 0} z^{\ell} (2\ell-1)!! [x^{2\ell}] e^{-\frac{1}{z} \left( e^x-\frac{x^2}{2} - x - 1 \right) + \frac12 x } \sum_{m \geq 0} z^{m} e^{-mx}(-1)^m\Ch_m \\
         &=  \sum_{m \geq 0}(-1)^m\Ch_m  z^{m} \sum_{\ell \geq 0} z^{\ell} (2\ell-1)!! [x^{2\ell}] e^{-\frac{1}{z} \left( e^x-\frac{x^2}{2} - x - 1 \right) + \left(\frac12-m\right) x},
    \end{align*}
where $\Ch_k$ is the coefficient of $z^k$ in the formal power series $\exp\left( \sum_{n\geq 1} \ch_n z^n \right)$.
By eq.\ \eqref{eqn:graph_stirling_asymp_psi}  this is
$$1 =  \sum_{m \geq 0} (-1)^m \Ch_m \sum_{k\geq m} c_{m,k} z^k=\sum_{k\geq 0}\sum_{m\leq k} (-1)^m \Ch_m c_{m,k} z^k.$$
 Because $[z^k] 1 = c_k$, we can also write this as $c_k =  \sum_{m \leq k} (-1)^m \Ch_m c_{m,k}$ for all $k\geq 0$. Therefore, we constructed a solution of the triangular equation system in \eqref{eqn:triangular_system}. %
Because the coefficients $c_m'$ are unique, it follows that $c_m' = (-1)^m \Ch_m$ as claimed.
\end{proof}

\begin{remark} The coefficients $c_{m,k}$ of the asymptotic expansion of the functions $\psi_{m}$ given in Proposition~\ref{prop:graph_stirling} (eq.\ \eqref{eqn:graph_stirling_asymp_psi}) can also be written in terms of Bernoulli numbers if we use the conventional expression of Stirling's approximation given in  eq.\ \eqref{eqn:stirling_approximation}. Slightly abusing the $\sim$ notation me may  write the asymptotic expansion for $\psi_m(n)$ as
\begin{align*}
\frac{\Gamma( n -m +\frac12 )}{\sqrt{2 \pi} e^{-n} n^{n}} &\sim 
\frac{\sqrt{2\pi} \left(n-m+\frac12\right)^{n-m} e^{-n+m-\frac12}\exp\left( \sum_{k\geq1} \frac{B_{k+1}}{k(k+1)} \left(n-m+\frac12\right)^{-k} \right)}{\sqrt{2 \pi} e^{-n} n^{n}}
\\
&=
n^{-m}\left(\frac{n-m+\frac12}{n}\right)^{n-m} e^{m - \frac12} \exp\left(\sum_{k\geq1} \frac{B_{k+1}}{k(k+1)} \left(n-m+\frac12\right)^{-k} \right).
\end{align*}
Writing $z=\frac{1}{n}$ this becomes
$$z^m\left(\left({1-z\left(m-\frac12\right)}\right)^{\frac{1}{z}-m}e^{m - \frac12} \exp\left(\sum_{k\geq1} \frac{B_{k+1}}{k(k+1)} \left(\frac{1}{z}-m+\frac12\right)^{-k} \right)\right).
$$
Since the coefficients of the asymptotic expansion for $\psi_m(n)$ are given by the above power series as well as  by the power series in eq.\ \eqref{eqn:graph_stirling_asymp_psi}, the series are equal, giving the following identity for Bernoulli numbers, for all $m\geq 0$. 
\begin{align*}
&\sum_{\ell \geq 0}
z^{\ell} (2\ell-1)!!
[x^{2\ell}]
\exp\left(-\frac{1}{z}\left(e^x - \frac{x^2}{2}  - x - 1\right) + x \left( \frac12 - m \right) \right) \\
&=\exp\left( 
 \left(m-\frac{1}{z}\right) \log \frac{1}{1-z\left(m-\frac12\right)} + m - \frac12 +\sum_{k\geq1} \frac{B_{k+1}}{k(k+1)} z^k \left(1-z\left(m-\frac12\right)\right)^{-k} \right) 
 \\
&=\exp\left( 
\sum_{k\geq1} \frac{z^k}{k(k+1)} \left( \left(m-\frac12 k \right) \left(m-\frac12\right)^k + \frac{B_{k+1}}{\left(1-z\left(m-\frac12\right)\right)^{k}} \right) \right)
\end{align*}

This identity actually holds for all $m\in \R$. However, it is unclear how to prove such an identity without asymptotic techniques. The special case $m=\frac12$ lies at the heart of the proof of Proposition~\ref{prop:bernoulli_graphs_sum}. De Bruijn also discusses this case using Laplace's method and writes that the identity is `by no means easy to verify directly' \cite[Sec.\ 4.5]{de1981asymptotic}.
\end{remark}
    \subsection{Laplace's method: A bridge between graphical enumeration and asymptotics}
\label{sec:laplace}
    A common source of asymptotic expansions is \textit{Laplace's method}. Laplace's method is, as one might guess from the name,  quite an old technique. It is usually used to extract asymptotic information from a complicated integral without evaluating it in full generality. We will use Laplace's method in the opposite way, as we are going analyze the properties of a complicated number sequence by associating it with a relatively simple integral. This way, the method will serve as a bridge between graphical enumeration as described in Section~\ref{sec:graphical_enumeration} and the analytic world of integrals and their asymptotic expansions.

\begin{lemma}[Laplace's method]
    \label{lmm:laplace_method}
Let $f$ and $g$ be real-valued functions  on a domain  $D \subset \R$  with $0$ in its interior.  Suppose both $f$ and $g$ are analytic in a neighborhood of $0$, that $g(0)=g'(0)=0$, $g''(0)=-1,$ and $0$ is the unique global supremum of $g$. Finally, assume that the integral
    \begin{align*}
 \int_D | f(x) | e^{n g(x)} dx
    \end{align*}
    exists for sufficiently large $n$.
  Then the sequence $I(n)$ given by the integral formula
\begin{align}
    \label{eqn:integral_I}
    I(n) = \sqrt{\frac{n}{2 \pi}} \int_D f(x) e^{n g(x)} dx
\end{align}
admits an asymptotic expansion with asymptotic scale $\{n^{-k}\}_{k\geq0}$, 
\begin{align}
    \label{eqn:laplace_expansion}
    I(n) \sim \sum_{k \geq 0} c_k n^{-k} \text{ as } n \rightarrow \infty,
\end{align}
where $c_k$ is the coefficient of $z^k$ in the formal power series,
\begin{align}
\label{eqn:laplace_coeffs}
 \sum_{\ell \geq 0} z^\ell (2\ell-1)!! [x^{2\ell}] f(x) e^{\frac{1}{z} \left( g(x) + \frac{x^2}{2} \right) }.
\end{align}
\end{lemma}
A quite similar statement is given in \cite[Thm. B7]{flajolet2009analytic}. Unfortunately, only a partial proof is given there. For the convenience of the reader we  provide a proof in the appendix. The argument revolves around approximating the integral in eq.\ \eqref{eqn:integral_I} with a Gaussian integral. It closely follows the arguments in \cite[Sec.\ 4.4]{de1981asymptotic} and \cite[Thm. B7]{flajolet2009analytic}.

We wrote the coefficients of the asymptotic expansion in eq.\ \eqref{eqn:laplace_coeffs}    suggestively to illustrate the close relationship of asymptotic expansions which come from Laplace's method 
and generating functions of graphs such as the one in Proposition~\ref{prop:convoluted_graph_sum}. We will use this relationship in the following Corollary, which we will need to give the relation between graphs and the zeta function stated in Proposition~\ref{prop:bernoulli_graphs_sum}.    
\begin{corollary}
    \label{crll:graph_laplace}
   Let  $f$ be the constant function $f(x)\equiv 1$, and assume $g$ is analytic near $0$ with  Taylor series $$g(x)=-\frac{x^2}{2} + \sum_{s\geq 3} x^s \frac{b_s}{s!}.$$
       Then  for all $k\geq 0$ the coefficients $c_k$ of the asymptotic expansion in eq.\ \eqref{eqn:laplace_coeffs} can be written as a weighted sum over graphs,
\begin{align}
    c_k = \sum_{ \substack{ \G \in \GG_0\\ |\G| = k} }\frac{ \prod_{v \in V(\Gamma)} b_{|v|} }{|\Aut \G|},
\end{align}
where $|v|$ is the \textit{valence} of the vertex $v$. \end{corollary}
\begin{proof}
    Let $\phi: \GG \rightarrow \R\llrrparen{z}$ be the function from  the set of graphs to the space of Laurent series in $z$ defined by setting  $\phi(\G) = 0$ if $\G$ contains an edge and $\phi(\G) = \prod_{v \in V(\G)} \left(z^{-1}b_{|v|}\right)$ if $\G$ has no edges. There are only finitely many graphs with $2\ell$ leaves which have no edges, and the function $\phi$ is multiplicative on the disjoint union of graphs, so we may apply Proposition~\ref{prop:convoluted_graph_sum} and Lemma~\ref{lmm:exponential_formula} to get
\begin{align*}
    \sum_{\G \in \GG_0}\frac{w^{e(\G)} \prod_{v \in V(\G)} (z^{-1}b_{|v|})}{|\Aut \G|} = \sum_{\ell \geq 0} w^\ell (2\ell-1)!! [x^{2\ell}] \exp\left( \sum_{\gamma \in \GG^c} x^{s(\gamma)} \frac{\phi(\gamma)}{|\Aut \gamma|} \right),
\end{align*}
where we used the fact that a graph has only one subgraph with no edges. The only graphs without edges which are also connected are the star graphs $R_{0,s}$. This together with the fact that $R_{0,s}$ has the symmetric group $\Sigma_s$ as automorphism group gives   
\begin{align*}
\sum_{\gamma \in \GG^c} x^{s(\G)} \frac{\phi(\gamma)}{|\Aut \gamma|} = \sum_{s\geq 3} x^s\frac{\phi(R_{0,s})}{|\Aut R_{0,s}|}
= \frac{1}{z} 
\sum_{s\geq 3} x^s \frac{b_s}{s!}.
\end{align*}
   Setting $w = z$ results in,
\begin{align*}
    \sum_{\G \in \GG_0}\frac{\prod_{v \in V(\G)} b_{|v|}}{|\Aut \G|}z^{|\G|} = \sum_{\ell \geq 0} z^{\ell} (2\ell-1)!! [x^{2\ell}] \exp\left( \frac{1}{z}
    \sum_{s\geq 3} x^s \frac{b_s}{s!}  \right). 
\end{align*}
The right hand side is now exactly the power series  given in eq.\ \eqref{eqn:laplace_coeffs} that  determines $c_k$.
\end{proof}

\begin{proof}[Proof of Proposition~\ref{prop:bernoulli_graphs_sum}]
We start with Euler's integral representation of the gamma function 
$$\G(n)=\int_{0}^\infty u^ne^{-u}\frac{du}{u}.$$  Substituting  $u = n e^x$ gives
\begin{align*}
    \Gamma( n ) &= \int_{-\infty}^\infty n^n e^{nx} e^{-ne^x}dx
    = e^{-n} n^{n} \int_{-\infty}^\infty  e^{-n\left(e^x-x - 1 \right) } dx.
\end{align*}
We can now apply Lemma~\ref{lmm:laplace_method} with $g(x) =-(e^x-x-1)$, $f(x)=1$ and $D=\R$ to get
 an asymptotic expansion
\begin{align*}
    \Gamma( n ) &= \sqrt{2\pi} e^{-n} n^{n-\frac12} \sum_{k \geq 0} c_k n^{-k}.
\end{align*}
By Corollary~\ref{crll:graph_laplace} and because $-(e^x-x-1) = -\sum_{s\geq3} \frac{x^s}{s!}$ the coefficients satisfy,
\begin{align*}
    c_k = \sum_{ \substack{ \G \in \GG_0\\ |\G| = k} }\frac{ (-1)^{v(\Gamma)} }{|\Aut \G|} \text{ for all } k \geq 0.
\end{align*}
Stirling's approximation in eq.\ \eqref{eqn:stirling_approximation} gives another expression for the coefficients $c_k$. Because the different ways to express the asymptotic expansion of $\Gamma(n)$ with the same scale and limit must coincide, we get
\begin{align*}
\sum_{ \substack{ \G \in \GG_0} }\frac{ (-1)^{v(\Gamma)} }{|\Aut \G|} z^{|\Gamma|}
=
\exp\left( \sum_{k=1}^\infty \frac{B_{k+1}}{k(k+1)} z^k \right).
\end{align*}
Since taking the formal logarithm restricts the sum on the left to connected graphs (Lemma~\ref{lmm:exponential_formula}) we get
\begin{align*}
\sum_{ \substack{ \G \in \GG_0^c} }\frac{ (-1)^{v(\Gamma)} }{|\Aut \G|} z^{|\Gamma|}
=
  \sum_{k=1}^\infty \frac{B_{k+1}}{k(k+1)} z^k.
\end{align*}
Now notice that 
$\sigma(\Gamma) = (-1)^{e(\Gamma)} = (-1)^{|\Gamma|} (-1)^{v(\Gamma)}$ and $B_{k+1} =0$ for all even $k > 0$, giving 
\begin{align*} 
\sum_{ \substack{ \G \in \GG_0^c\\ |\G| = n} }\frac{ \sigma(\Gamma) }{|\Aut \G|}&= (-1)^n\frac{B_{n+1}}{n(n+1)}= - \frac{B_{n+1}}{n(n+1)} =\frac{\zeta(-n)}{n}. \qedhere
\end{align*}
\end{proof}

We now turn to the proof of Proposition~\ref{prop:graph_stirling}, which follows along similar lines.
\begin{proof}[Proof of Proposition~\ref{prop:graph_stirling}]
    Assume $n,m \geq 0$ with $n \geq \max\{1,m\}$. Start with Euler's integral and substitute $u = n e^x$ to obtain
\begin{align*}
    \Gamma\left( n -m +\frac12 \right) &= \int_0^\infty u^{n-m+\frac12} e^{-u} \frac{du}{u} 
    = e^{-n} n^{n-m+\frac12} \int_{-\infty}^\infty  e^{-n\left(e^x-x - 1 \right) + x \left( \frac12 - m \right) } dx.
\end{align*}
Therefore,
\begin{align*}
\psi_m (n)= 
\frac{\Gamma( n -m +\frac12 )}{\sqrt{2 \pi} e^{-n} n^{n} } = 
n^{-m}\sqrt{\frac{n}{2\pi}} 
\int_{-\infty}^\infty  e^{-n\left(e^x-x-1 \right) + x \left( \frac12 - m \right) } dx.
\end{align*}
The condition  $n \geq \max\{1,m\}$ guarantees that the integral exists.  
The functions $f(x) =e^{ x \left( \frac12 - m \right)}$ and $g(x) = -(e^x-x - 1)$,  defined on $D= \R$, satisfy the conditions of Lemma~\ref{lmm:laplace_method}, so we can apply Laplace's method to obtain
 \begin{align*}
    n^m \psi_m(n) \sim \sum_{k \geq 0} c_{m,k}' n^{-k} \text{ as } n \rightarrow \infty,
\end{align*}
where $c'_{m,k}$ is the coefficient of $z^k$ in the power series
\begin{align*}
\sum_{\ell \geq 0} z^{\ell} (2\ell-1)!! [x^{2\ell}] e^{- \frac{1}{z} \left(e^x- \frac{x^2}{2} -x - 1 \right) + x \left( \frac12 - m \right)}.
\end{align*}
From Definition \ref{def:asymptotic_expansion} and  the fact that $n^{-m}\bigO(n^{-R+m}) = \bigO(n^{-R})$, it follows that
\begin{align*}
    \psi_m(n) \sim \sum_{k \geq m} c_{m,k-m}' n^{-k} \text{ as } n \rightarrow \infty.
\end{align*}
Setting $c_{m,k} := c_{m,k-m}'$ gives eq.\ \eqref{eqn:graph_stirling_asymp_psi}.
\end{proof}

We have now completed all of the steps in the proof  of Theorem~\ref{thm:asymptotic_expansion}.
Before we continue with the proof of Theorem~\ref{thm:SVconj}, we will briefly discuss the relationship of our considerations with Kontsevich's Lie graph complex.
\subsection{Lie graph complex}
\label{sec:kontsevich} Kontsevich's {\em Lie graph complex} $\mathfrak{L}_*$  computes the Chevalley-Eilenberg homology of a certain infinite-dimensional Lie algebra $\ell_\infty$ associated to the Lie operad. %
In \cite{kontsevich1993formal} Kontsevich remarked that the orbifold Euler characteristic of the subcomplex $\mathfrak{L}^{(n)}_*$ spanned by connected graphs with fundamental group $F_n$ can be encoded as coefficients of the asymptotic expansion of the integral
\begin{align}
    \label{eqn:integral_kontsevich}
    \sqrt{\frac{n}{2 \pi}} \int_{D} \exp\left(-n \sum_{s \geq 2} \frac{x^s}{s(s-1)}\right) \sim
\sum_{k \geq 0} c_k n^{-k} \text{ as } n \to \infty,
\end{align}
where $D$ is a small domain that contains a neighborhood of $0$ and $c_k$ is the $z^k$ coefficient of the power series $\exp( \sum_{n \geq 1} \chi(\mathfrak{L}_*^{(n+1)}) z^n )$. %
Observing that $- \sum_{s \geq 2} \frac{x^s}{s(s-1)} = -\sum_{s \geq 2} (s-2)! \frac{x^s}{s!}$ and using Corollary~\ref{crll:graph_laplace} together with the exponential formula (Lemma~\ref{lmm:exponential_formula}), we may conclude that 
\begin{align*}
    \chi(\mathfrak{L}_*^{(n)}) = \sum_{\substack{\G \in \GG^c_0\\ \pi_1(\G) \cong F_n}} \frac{\xi(\G)}{|\Aut \G|},
\end{align*}
where $\xi$ is the function given by $$\xi(\G) = (-1)^{|V(\G)|} \prod_{v \in V(\G)} (|v|-2)!$$ This formula for $\chi(\mathfrak{L}_*^{(n)})$ also follows directly from counting graphs whose vertices are dressed with Lie operad elements. 
We have $\chi(\mathfrak{L}_*^{(n)}) = \ch_{n-1}$, because
\begin{align*}
    H_k(\mathfrak{L}_*^{(n)}) \iso H^{2n-2-k}(\Out(F_n)).
\end{align*}
This was first observed by Kontsevich \cite{kontsevich1993formal}; see \cite{conant2003theorem} for a detailed proof. The statements in Theorems~\ref{thm:SVconj} and \ref{thm:asymptotic_expansion}, therefore apply verbatim to the orbifold Euler characteristic of $\mathfrak{L}_*^{(n)}$. It is, however, unclear what role the map $\xi$ and the Lie graph complex play in the interesting Hopf algebraic duality between $\tau$ and $\sigma$ explained in Section~\ref{sec:core}. 

The integral in eq.\ \eqref{eqn:integral_kontsevich} gives another representation of the coefficients $\ch_n,$ but the descriptive power of this representation is limited: it seems that the integral does not evaluate to a `known' function, which could facilitate the extraction of information about the coefficients $\ch_n$. Recall that the fact that two functions have the same asymptotic expansion does not imply their equality, so it does not follow from the considerations above and Theorem~\ref{thm:asymptotic_expansion} that the left hand side of eq.\ \eqref{eqn:integral_kontsevich} is equal to $\sqrt{2 \pi}e^{-n} n^n$. %

\section{The Lambert \texorpdfstring{$W$}{W}-function}
In this section we prove that  the coefficients of the asymptotic expansion in Theorem~\ref{thm:asymptotic_expansion} are all negative. The first statement of Theorem~\ref{thm:SVconj},   that $\chi(\Out(F_n)) < 0$ for all $n \geq 2$, follows then by Lemma~\ref{lmm:exp_negative}.

\subsection{Singularity analysis}

We will accomplish this by using a second method to obtain the asymptotic expansion of the sequence $\sqrt{2 \pi} e^{-n} n^n$ with respect to the asymptotic scale $\{(-1)^k\Gamma(n-k+\frac{1}{2})\}_{k\geq0}$. This second method is \textit{singularity analysis}. By Theorem~\ref{thm:asymptotic_expansion} and because of the uniqueness of asymptotic expansions, we therefore obtain another expression for the coefficients $\hat\chi_n$ of $\exp(\sum_{n\geq 1} \ch_n z^n)$. This expression will involve the Lambert $W$-function, which is defined as the solution of the functional equation $W(z) e^{W(z)}= z$ \cite{corless1996lambertw}. Eventually, we will use a theorem of Volkmer \cite{volkmer2008factorial} to show that the coefficients of the asymptotic expansion are negative. 

\begin{figure}
\begin{tikzpicture} \begin{axis}[ height=\figureheight, tick align=outside, tick pos=left, width=\figurewidth, x grid style={white!69.01960784313725!black}, xlabel={\(\displaystyle z\)}, xmajorgrids, xmin=-1, xmax=1, xtick style={color=black}, y grid style={white!69.01960784313725!black}, ylabel={\(\displaystyle W(z)\)}, ymajorgrids, ymin=-5, ymax=5, ytick style={color=black} ] \addplot [semithick, black] table {%
-0.367879441171442 -1
-0.367322540037808 -0.945960576192766
-0.365569562332727 -0.891921152385532
-0.362489397863951 -0.837881728578298
-0.357940064029995 -0.783842304771064
-0.351767902646442 -0.72980288096383
-0.343806721099256 -0.675763457156596
-0.333876874118769 -0.621724033349363
-0.321784282228105 -0.567684609542129
-0.307319382664605 -0.513645185734895
-0.290256008301518 -0.459605761927661
-0.270350189808666 -0.405566338120427
-0.247338875984067 -0.351526914313193
-0.220938566862327 -0.297487490505959
-0.19084385385887 -0.243448066698725
-0.156725860840461 -0.189408642891491
-0.118230579620611 -0.135369219084257
-0.0749770929619213 -0.0813297952770232
-0.0265556777246851 -0.0272903714697893
0.0274742196695564 0.0267490523374447
0.0875861437909174 0.0807884761446787
0.154288935633532 0.134827899951913
0.228129192355259 0.188867323759146
0.309693891352915 0.24290674756638
0.399613189339125 0.296946171373614
0.498563407764693 0.350985595180848
0.607270216650625 0.405025018988082
0.72651202965913 0.459064442795316
0.857123624045936 0.51310386660255
1 0.567143290409784
}; \addplot [semithick, black, dashed] table {%
-0.0336897349954273 -5
-0.0376055021646619 -4.86206896551724
-0.0419426179058552 -4.72413793103448
-0.0467400631717225 -4.58620689655172
-0.0520391328982212 -4.44827586206897
-0.0578832678204224 -4.31034482758621
-0.0643177859261937 -4.17241379310345
-0.0713894875400297 -4.03448275862069
-0.0791461025317493 -3.89655172413793
-0.0876355415898728 -3.75862068965517
-0.0969049056948192 -3.62068965517241
-0.106999198646334 -3.48275862068966
-0.117959676476869 -3.3448275862069
-0.129821754505718 -3.20689655172414
-0.142612377291115 -3.06896551724138
-0.156346738389076 -2.93103448275862
-0.171024215124495 -2.79310344827586
-0.186623357930594 -2.6551724137931
-0.203095743525044 -2.51724137931034
-0.220358465453742 -2.37931034482759
-0.238284993397005 -2.24137931034483
-0.256694082986828 -2.10344827586207
-0.275336359427955 -1.96551724137931
-0.293878129430117 -1.82758620689655
-0.31188189506695 -1.68965517241379
-0.328782948102461 -1.55172413793103
-0.343861311642715 -1.41379310344828
-0.356208164845436 -1.27586206896552
-0.364685732545771 -1.13793103448276
-0.367879441171442 -1
}; \addplot [semithick, black, dotted] table {%
-0.367879441171442 -5
-0.367879441171442 5
}; \addplot [semithick, black, dotted] table {%
0 -5
0 5
}; \end{axis} \end{tikzpicture}
\caption{Plot of the two real branches of the Lambert $W$-function. The solid line depicts the principal branch $W_0$, the dashed line the other real branch, $W_{-1}$. Both branches share a square root type singularity at $z = -1/e$. The $W_{-1}$ additionally has a logarithmic singularity at $z=0$. The locations of the singularities are indicated with dotted lines.
}
\label{fig:lambertW}
\end{figure}
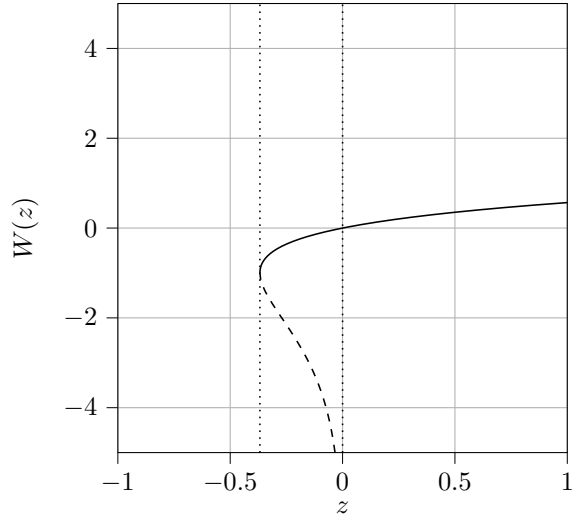

\begin{proposition}
    \label{prop:lambert_rep}
 The coefficient $\Ch_k$  of $z^k$ in $\exp\left(\sum_{n\geq 1} \ch_n z^n\right)$ satisfies
\begin{align}
    \Ch_k = -2 \frac{\Gamma(k +\frac12 )}{\sqrt{2\pi}} v_{2k-1} \text{ for all } k\geq 0,
\end{align}
where the $\{v_k\}_{k\geq -1}$ are the coefficients of the following expansion involving the derivative of the principal branch of the Lambert $W$-function in the vicinity of its branch-point at $z=-\frac{1}{e}$, 
\begin{align}
    \label{eqn:ck_def_lambert}
    z W_0'(z) &= \sum_{k= -1}^{\infty} (-1)^{k+1} v_{k} (1+ez)^{\frac{k}{2}}.
\end{align}
\end{proposition}
In Figure~\ref{fig:lambertW}, the principal branch $W_0$ of the Lambert $W$-function is depicted with a solid line. Note that the index $k$ in the summation starts with $-1$. We chose this notation to be consistent with Volkmer \cite{volkmer2008factorial}, who proved a couple of interesting properties of the numbers $v_k$ motivated by a problem posed by Ramanujan. Most important for our considerations, he shows in \cite[Thm.\ 3]{volkmer2008factorial} that $v_k > 0$ for all $k \geq 1$.  He proves this by deriving the following integral representation for the coefficients $v_k$ \cite[Thm.\ 2]{volkmer2008factorial},
\begin{align*}
    v_k &= - \frac{1}{2\pi} \int_0^\infty (1+z)^{-\frac{k}{2}-1} \frac{\Im W_{-1}(e^{-1} z)}{|1+W_{-1}(e^{-1}z)|^2} dz \text{ for  all }k \geq 1,
\end{align*}
where  $\Im$ denotes the imaginary part of a complex number and $W_{-1}$ is the branch of the Lambert $W$-function which is real and decreasing on the interval $(-\frac{1}{e}, 0)$. This branch is drawn with a dashed line in Figure~\ref{fig:lambertW}. The integrand is strictly negative since $\Im W_{-1}(z) \in (-2\pi, -\pi)$ for $z \in (0,\infty)$ \cite{corless1996lambertw}.

\begin{corollary}
    \label{crll:negative_T}
    For all $n\geq 2$, $\chi\left( \Out(F_{n}) \right) < 0$.
\end{corollary}
\begin{proof}
    Apply Proposition~\ref{prop:lambert_rep}, the fact that $\Gamma(k+\frac12) > 0$ and \cite[Thm.\ 3]{volkmer2008factorial} to get 
$\Ch_n < 0$ for all $n \geq 1$. The result now follows from  Lemma~\ref{lmm:exp_negative}.
\end{proof}

As already mentioned, we will use singularity analysis to prove Proposition~\ref{prop:lambert_rep}. The basic observation behind singularity analysis is the following: the radius of convergence of the Taylor expansion of a function $f$ is equal to norm of  the singularity of $f$ in $\C$ which is closest to the origin. This singularity is called the \textit{dominant singularity} of the function. The radius of convergence is also equal to the limit $\limsup_{n\rightarrow \infty} |a_n|^{-\frac{1}{n}}$  where $f(z) = \sum_{n=0}^\infty a_n z^n$. The radius of convergence  therefore determines the exponential growth rate of the coefficients $a_n$. In many cases, the detailed nature of the function's dominant singularity determines the asymptotic behaviour of the coefficients completely. To illustrate these notions, we will start by proving one of the most basic statements from the framework of singularity analysis. For other required statements from this framework, we will refer to the literature. 
A very detailed and instructive introduction to singularity analysis can be found in Flajolet's and Sedgewick's book \cite[Part 2]{flajolet2009analytic}. 
\begin{lemma}
    \label{lmm:dominant_singularity}
    If $g$ is a generating function with power series expansion $g(z)= \sum_{n=0}^\infty b_n z^n,$ which has radius of convergence $r$, then 
\begin{align*}
    b_n \in \smallO\left(C^{-n}\right) \text{ for all } 0< C < r.
\end{align*}
\end{lemma}
\begin{proof}
By elementary calculus, $r^{-1} = \limsup_{n \rightarrow \infty} |b_n|^{\frac{1}{n}}.$ Therefore for every $\delta > 0$ there exists an $n_0$  such that $|b_n|^{\frac{1}{n}} < r^{-1}+\delta$ for all $n \geq n_0.$ It follows that $$|b_n| < (r^{-1} +\delta)^n = \left( \frac{r}{1 + \delta r} \right)^{-n} \text{ for all } n \geq n_0.$$
Because we can choose any $\delta > 0$, the statement follows. This argument also works if $r = \infty$. 
\end{proof}

Suppose we can decompose a generating function $h(z)=\sum_{n\geq 0}d_nz^n$ as a sum $h(z)=f(z)+g(z)$ with $f(z)=\sum_{n\geq 0}a_nz^n$ and $g$ analytic in a disk around $0\in \C$ of radius larger than $1$.
Then by Lemma~\ref{lmm:dominant_singularity} there is a constant $C>1$ such that 
\begin{align*}
    d_n = a_n + \smallO(C^{-n}).
\end{align*}
This is especially useful if the coefficients $a_n$ have an asymptotic expansion,
\begin{align*}
    a_n \sim \sum_{k \geq 0 } c_k \varphi_k(n) \text{ as } n \to \infty,
\end{align*}
with an asymptotic scale $\{\varphi_k\}_{k \geq 0}$ which satisfies $\smallO(C^{-n}) \subset \bigO(\varphi_k(n))$ for all $k \geq 0$. In this common case, we may neglect terms contributed by $g$ to the generating function $h$
and conclude that 
\begin{align*}
    d_n \sim \sum_{k \geq 0 } c_k \varphi_k(n) \text{ as } n \to \infty.
\end{align*}

\begin{figure}
\begin{center}
\begin{tikzpicture} [scale=1.5] \fill [gray!20] ([shift=(20:1.3cm)] 0,0) arc (20:340:1.3); \fill [white] (20:1.3) to (.5,0) to (340:1.3) to (20:1.3); \draw [thick] ([shift=(0:.8)] 0,0) arc (0:20:.5); \draw[->] (-2,0) to (2,0); \draw[->] (0,-2) to (0,2); \fill (0:.5) circle (.025); \node [above] (one) at (0:.5) {$1$}; \node [below] (Re) at (2,0) {$\Re z$}; \node [left] (Im) at (0,2) {$\Im z$}; \node (fi) at (1,.15) {$\phi$}; \node (Delta) at (-.4,.4) {$\Delta$}; \end{tikzpicture}
\caption{The region $\Delta \subset \C$ in the statement of Lemma~\ref{lmm:singularity_analysis}}\label{fig:pacman}
\end{center}
\end{figure}
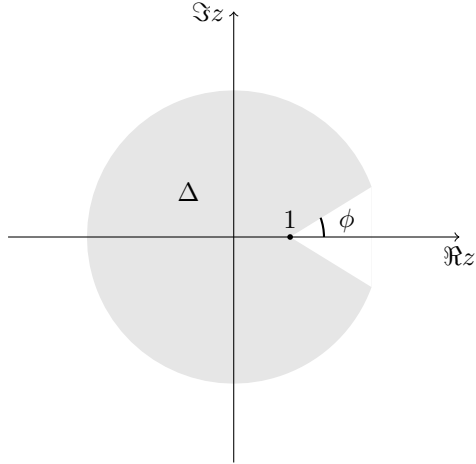

To prove Proposition~\ref{prop:lambert_rep} we will need
\begin{lemma}[Basic singularity analysis {\cite[Cor.\ 3]{flajolet1990singularity}}]
    \label{lmm:singularity_analysis}
    Let $f:\C \rightarrow \C$ be analytic at $0$ with an isolated singularity at $1$, such that $f(z)$ can be analytically continued to an open domain of the form $\Delta = \{ z : |z| < R, z \neq 1, | \arg(z-1)| > \phi \} \subset \C$ with some $R > 1$ and $0 < \phi < \pi/2$ (see Figure~\ref{fig:pacman}). 
    If $f(z)$ has the following asymptotic behaviour in $\Delta$ for $R \geq 0$,
    \begin{align}
        \label{eqn:singular_limit_exp}
        f(z) &= \sum_{k=0}^{R-1} c_k (1-z)^{\alpha_k} + \bigO\left((1-z)^{A}\right) \text{ as } z \rightarrow 1^{-},
    \end{align}
    where $c_k\in \R$ and $\alpha_0 \leq \alpha_1 \leq \ldots \leq \alpha_{R-1} < A \in \R$, then the coefficients $a_n = [z^n] f(z)$ have the asymptotic behaviour,     \begin{align}
        \label{eqn:singular_limit_asymp}
        a_n = \sum_{k = 0}^{R-1} c_k \binom{ n - \alpha_{k} -1}{n} + \bigO(n^{-A-1}) \text{ as } n \rightarrow \infty.
    \end{align}
\end{lemma}
    Note that eq.\ \eqref{eqn:singular_limit_asymp} is not an asymptotic expansion in the sense of Definition \ref{def:asymptotic_expansion}, because we did not specify an asymptotic scale. 

\begin{proof}[Proof of Proposition~\ref{prop:lambert_rep}]
The principal branch of the Lambert $W$-function has the series representation \cite{corless1996lambertw},
\begin{align*}
    W_0(z) &= \sum_{n\geq1} (-1)^{n+1} \frac{n^{n-1}}{n!} z^n.
\end{align*}
By acting with $z \frac{d}{d z}$, we obtain the expansion
\begin{align}
    \label{eqn:Wprime_expansion}
    z W_0'(z) &= \sum_{n\geq1} (-1)^{n+1} \frac{n^{n}}{n!} z^n.
\end{align}
The function $W_0$ is analytic in the cut plane $\C \setminus [-1/e,-\infty)$ and has an expansion in the vicinity of the branch point at $z=-1/e$,
\begin{align*}
    W_0(z) &= -1 + \sqrt{2(1+ez)} - \frac{2}{3} (1+ez) + \frac{11}{72} \left(\sqrt{2(1+ez)}\right)^{3} + \ldots
\end{align*}
which is convergent if $z \in [-1/e,0)$ \cite[Sec.\ 4]{corless1996lambertw} (see also Figure~\ref{fig:lambertW}).
Therefore, the function $z W_0'(z)$ has an expansion of the form
\begin{align*}
    z W_0'(z) &= \sum_{k = -1}^\infty (-1)^{k+1} v_{k} (1+ez)^{\frac{k}{2}}.
\end{align*}
Using the basic version of singularity analysis from Lemma~\ref{lmm:singularity_analysis}, we can obtain the asymptotic behaviour of the sequence $e^{-n} \frac{n^{n}}{n!}$ from this: 
first we rescale the $z$-variable of $z W_0'(z)$ to obtain the expansion,
\begin{align*}
    -\frac{z}{e} W_0'\left(-\frac{z}{e} \right) &= \sum_{k = -1}^{R-1} (-1)^{k+1} v_k (1-z)^{\frac{k}{2}} + \bigO\left((1-z)^{\frac{R}{2}}\right) \text{ as } z \rightarrow 1^{-} \text{ for all } R \geq 0.
\end{align*}
As $z W_0'(z)$ is analytic in the cut plane $\C \setminus [-1/e,-\infty)$, the function $-\frac{z}{e} W_0'\left(-\frac{z}{e} \right)$ is analytic in another cut plane $\C \setminus [1,\infty)$. 
As $\Delta \subset \C \setminus [1,\infty)$, we can apply Lemma~\ref{lmm:singularity_analysis} and eq.\ \eqref{eqn:Wprime_expansion} to get
\begin{align*}
    [z^n]\frac{-z}{e} W_0'\left(-\frac{z}{e} \right) &= 
    -e^{-n} \frac{n^{n}}{n!} =
    \sum_{k = -1}^{R-1} (-1)^{k+1} v_{k} \binom{n-\frac{k}{2} -1}{n}  + \bigO\left(n^{-\frac{R}{2}-1}\right) \text{ for all } R \geq 0,
\end{align*}
where we used $\alpha_k = \frac{k}{2}$ and $A = \frac{R}{2}$.
The even contributions in the sum over $k$ vanish since the first argument of the binomial coefficient is an integer that is smaller than the second. Therefore, 
\begin{align*}
    -e^{-n} \frac{n^{n}}{n!} &=  \sum_{k = 0}^{R-1} v_{2k-1} \binom{n-k -\frac12}{n}  + \bigO\left(n^{-R- \frac12}\right)  \text{ for all } R \geq 0.
\end{align*}
The binomial coefficient can be expressed in terms of $\Gamma$ functions $\binom{n-k -\frac12}{n} = \frac{\Gamma(n-k +\frac12)}{n! \Gamma(\frac12 - k)}$. As a consequence of the reflection formula $\Gamma(z)\Gamma(1-z) =\frac{\pi}{\sin(\pi z)}$, we have $\Gamma\left(\frac12 - k\right) = \frac{(-1)^k \pi}{\Gamma(k + \frac12)}$. Hence,
\begin{align*}
    -e^{-n} \frac{n^{n}}{n!} &=  \frac{1}{\pi n!}
    \sum_{k = 0}^{R-1}(-1)^k v_{2k-1} \Gamma\left( n - k + \frac12  \right)  \Gamma\left(k+\frac12\right) + \bigO\left(n^{-\frac{R}{2}-1}\right) \text{ for all } R \geq 0.
\end{align*}
The statement follows from the uniqueness of asymptotic expansions and the property of the $\Gamma$ function that $\bigO\left((n!) n^{-R-\frac12}\right) = \bigO\left(\Gamma\left(n-R+\frac12\right)\right)$.
\end{proof}

    Proposition~\ref{prop:lambert_rep} together with known techniques for evaluating the various expansion coefficients of the Lambert $W$-function provides an efficient way to calculate the numbers $\ch_n$: 

\begin{proposition}
\label{prop:efficient_chn}
The numbers $\ch_n$ and $\Ch_n$ can be calculated using the recursion equations,
    \begin{gather*}
     \ch_n=\Ch_n-\frac{1}{n}\sum_{k=1}^{n-1}k \ch_k \Ch_{n-k} \\
    \Ch_n = - (2n-1)!!\left( \frac12 (2n-1) \mu_{2n-1} -(2n+1) \mu_{2n+1} \right)
\\
    \mu_n = \frac{n-1}{n+1}\left( \frac{\mu_{n-2}}{2} + \frac{\alpha_{n-2}}{4} \right) -\frac{\alpha_{n}}{2} - \frac{\mu_{n-1}}{n+1} \\
    \alpha_n = \sum_{k=2}^{n-1} \mu_k \mu_{n+1-k},
    \end{gather*}
 for all $n \geq 1$ with $\alpha_0=2, \alpha_1= -1, \mu_{-1} = 0, \mu_0 = -1, \mu_1 = 1$ and $\Ch_0 = 1$.
\end{proposition}
\begin{proof}
The coefficients $\mu_n$ are the expansion coefficients of the Lambert-$W$ function near its branch point:
$W_0(z) = \sum_{n \geq 0} \mu_n \left( 2 ( 1+ez) \right)^{\frac{n}{2}}.$ The recursion for $\mu_n$ is given in \cite[eqs.\ (4.23) and (4.24)]{corless1996lambertw}; it follows from the differential equation which $W$ satisfies. 
We can adapt \cite[eq.\ (2.11)]{volkmer2008factorial} to the notation of \cite{corless1996lambertw} (compare \cite[eq.\ (2.1)]{volkmer2008factorial} with the definition of $\mu_n$) to get an expression for $v_n$ in terms of $\mu_n$:
\begin{align*}
        v_n =  (-1)^{n+1} 2^{\frac{n}{2}} \left( \frac12 n \mu_n -(n+2) \mu_{n+2} \right)
\end{align*}
The equation for $\Ch_n$ follows using Proposition~\ref{prop:lambert_rep} and $(2n-1)!! = 2^{n+\frac12}\Gamma(n+\frac12)$. Finally, we use eq.~\eqref{eqn:exp_pwrsrs_relation} to translate from $\Ch_n$ to $\ch_n$.
\end{proof}

    Written in power series notation with $T(z) = \sum_{n\geq 1} \ch_n z^n$ and $\exp(T(z)) = \sum_{n\geq 0} \Ch_n z^n$, the first few coefficients are
    \begin{gather*}
T(z)= - \frac{1}{24} z - \frac{1}{48} z^{2} - \frac{161}{5760} z^{3} - \frac{367}{5760} z^{4} - \frac{120257}{580608} z^{5}  + \ldots \\
\exp(T(z))= 1 - \frac{1}{24} z - \frac{23}{1152} z^{2} - \frac{11237}{414720} z^{3} - \frac{2482411}{39813120} z^{4} - \frac{272785979}{1337720832} z^{5}  + \ldots
    \end{gather*}
    With this approach we calculated the value of $\ch_n$ up to $n=1000$ with basic computer algebra tools. 
    
In addition to being able compute the value of $\ch_n$  for very large $n$,  we can also determine the explicit asymptotic behavior of the coefficients for large $n$.  We  do that in the next section.

\subsection{The asymptotic growth of \texorpdfstring{$\chi(\Out(F_n))$}{chi(Out(Fn))}}

\begin{proposition}
    \label{prop:chi_outfn_asymptotics}
    The Euler characteristic of $\Out(F_n)$ has the leading asymptotic behaviour,
    \begin{align}
        \chi(\Out(F_n))=
        - \frac{1}{\sqrt{2\pi}} \frac{\Gamma(n -\frac32 )}{\log^2 n}  + \bigO\left( \frac{\log \log n }{\log^3 n}\Gamma\left(n -\frac32 \right) \right) \text{ as } n\to \infty.
    \end{align}
\end{proposition}
We will prove Proposition~\ref{prop:chi_outfn_asymptotics} by applying a stronger version of singularity analysis to determine the asymptotic behaviour of the coefficients $v_k$. Proposition~\ref{prop:lambert_rep} and a classic theorem by Wright \cite{wright1970asymptotic} will eventually enable us to deduce the asymptotic behaviour of the sequence $\chi(\Out(F_n))$.
\begin{lemma}
    \label{lmm:asymp_vk}
    The coefficients $v_{k}$ have the leading asymptotic behaviour,
    \begin{align}
        v_{k} &=  -\frac{1}{k(\log k)^2} + \bigO\left( \frac{\log \log k}{k(\log k)^3}\right) \text{ as } k \to \infty.
    \end{align}
\end{lemma}
\begin{proof}
    In addition to the expansion in eq.\ \eqref{eqn:ck_def_lambert}, the numbers $v_k$ are the coefficients of the following expansion of the other real branch of the Lambert $W$-function \cite{volkmer2008factorial},  
    \begin{align*}
        z W_{-1}'(z) &= -\sum_{k = -1}^{\infty}v_k (1+ez)^{\frac{k}{2}} \text{ for } z \in \left(-\frac{1}{e},0\right).
    \end{align*}
The discrepancy between the two expansions is given by the two different choices for the branch of the square root. We first consider the odd coefficients $v_{2k-1}$. Setting $w = 1+ez$ we define
    \begin{align*}
        g(w)= \frac12 \sqrt{w} \left( z W_0'\left(z\right) - z W_{-1}'\left(z\right) \right) = \sum_{k = 0}^{\infty}v_{2k-1} w^{k}.
    \end{align*}
    The function $g(w)$ can be analytically continued to $w=0$. Moreover, $g(w)$ has no other singularities in a $\Delta$-domain as defined in Lemma~\ref{lmm:singularity_analysis}: the dominant singularity of $g(w)$ comes from the logarithmic singularity of $W_{-1}$ at $z=0$ (see Figure~\ref{fig:lambertW}), so is located at $w=1$ after the variable change. The principal branch $W_0$ is analytic at $z=0$. Neither $W_0$ nor $W_{-1}$ has any other singularities in the relevant domain.

    Because the differential equation $W'(z) = \frac{W(z)}{z(1+W(z))}$ is satisfied by every branch of the Lambert $W$-function, we have
    \begin{align*}
        g(w) &= \frac12 \sqrt{w} \left( \frac{W_{0}(z)}{1+W_{0}(z)} - \frac{W_{-1}(z)}{1+W_{-1}(z)} \right) = 
        - \frac12 \sqrt{w} \frac{W_{-1}\left(\frac{w-1}{e}\right)}{1+W_{-1}\left(\frac{w-1}{e}\right)} + \text{`analytic'} \text{ as } w\rightarrow 1^- \\
        &=
        \frac12 \frac{\sqrt{w}}{1+W_{-1}\left(\frac{w-1}{e}\right)} + \text{`analytic'} \text{ as } w\rightarrow 1^{-},
    \end{align*}
    where we are able to neglect contributions which are analytic at $w=1$ since, by Lemma~\ref{lmm:dominant_singularity}, they will eventually  contribute only exponentially suppressed terms asymptotically.
    The function $W_{-1}$ has the singular behaviour \cite[Sec.\ 4]{corless1996lambertw},
    \begin{align*}
        W_{-1}(z) = \log(-z) + \bigO\left(\log(-\log(-z))\right) \text{ as } z \rightarrow 0^{-}.
    \end{align*}
    Thus, we get the singular expansion for $g(w)$,
    \begin{align*}
        g(w) &= 
        \frac12 \frac{\sqrt{ 1- (1-w)}}{1+\log\left(\frac{1-w}{e} \right)+ \bigO\left(\log(-\log(\frac{1-w}{e}))\right)}+ \text{`analytic'} \text{ as } w\rightarrow 1^{-}
\\
            &= 
- \frac12 \left(\log\frac{1}{1-w}\right)^{-1} + \bigO\left( \frac{\log(-\log\left(1-w \right))}{\left(\log\left(1-w \right)\right)^2}\right) + \text{'analytic'} \text{ as } w\rightarrow 1^{-}.
    \end{align*}
    With this knowledge we may use a more general statement from singularity analysis to extract the asymptotics of the coefficients of $g(w)$, for instance \cite[Cor.\ 6]{flajolet1990singularity}. More details are given in \cite[Sec.\ VI.2]{flajolet2009analytic}, where one can find the `asymptotic transfer law' $[w^k] \left(\log\frac{1}{1-w} \right)^{-1}= -\frac{1}{k(\log k)^2} + \bigO\left(\frac{1}{k(\log k)^3}\right)$ for $k\rightarrow \infty$ in Table VI.5. Also `transferring' the $\bigO$ term in the singular expansion of $g$ into its corresponding asymptotic term for the coefficients \cite[Cor.\ 6]{flajolet1990singularity} gives,
    \begin{align*}
        [w^k]g(w)&= v_{2k-1} = \frac12 \frac{1}{k (\log k)^2} + \bigO\left( \frac{\log\log k}{k (\log k)^3} \right) \text{ as } k \rightarrow \infty.
    \end{align*}
 We note that the asymptotic behaviour of the even coefficients $v_{2k}$ follows analogously by starting with 
        $g(w)= \frac12 \left(- z W_0'\left(z\right) - z W_{-1}'\left(z\right) \right) = \sum_{k = 0}^{\infty}v_{2k} w^{k}$,  although we will not need this for the present article.
        \end{proof}

The only remaining task for proving   Theorem~\ref{thm:SVconj} is to transfer our knowledge of the asymptotic behaviour of $v_k$ to the coefficients $\ch_n$. To deduce the asymptotic behaviour of these coefficients, we will use a classical theorem by Wright in the theory of graphical enumeration.
\begin{lemma}[Thm.\ 2 of \cite{wright1970asymptotic} with $R=1$] 
    \label{lmm:wright_connected_asymptotics}
Let $f(x)= \sum_{n \geq 0} c_n x^n$ be a power series in $\R[[x]]$, and let  $\exp(f(x)) = \sum_{n \geq 0} \hat{c}_n x^n$. Suppose 
    $c_0 = 0$, $\hat{c}_0 = 1,$  and  $\hat{c}_{n-1} \in \smallO(\hat{c}_{n})$  as $n\to\infty$ as well as 
    \begin{align}
        \label{eqn:center_sum}
        \sum_{k = 1}^{n-1} \hat{c}_k \hat{c}_{n-k} \in \bigO(\hat{c}_{n-1}) \text{ as } n\to \infty.
    \end{align}
Then 
    $c_n = \hat{c}_n + \bigO(\hat{c}_{n-1})$ as $n \to \infty$.
\end{lemma}

\begin{proof}[Proof of Proposition~\ref{prop:chi_outfn_asymptotics}]
Let $T(z) = \sum_{n\geq 1} \ch_n z^n$, and $\exp(T(z)) = \sum_{n\geq 0} \Ch_n z^n$.

    We have to verify that $\Ch_n$ satisfies the conditions of Lemma~\ref{lmm:wright_connected_asymptotics}. The only condition that is not immediate is eq.\ \eqref{eqn:center_sum}. By Proposition~\ref{prop:lambert_rep} and Lemma~\ref{lmm:asymp_vk} we have
    \begin{align*}
        \Ch_n &= - \frac{1}{\sqrt{2\pi}} \frac{\Gamma(n +\frac12 )}{n \log^2 n}  + \bigO\left( \frac{\log \log n }{n \log^3 n}\Gamma(n +\frac12 ) \right)\\
            &= - \frac{1}{\sqrt{2\pi}} \frac{\Gamma(n -\frac12 )}{\log^2 (n+1)}  + \bigO\left( \frac{\log \log n }{\log^3 n}\Gamma(n -\frac12 ) \right).
    \end{align*}
  From this it follows that we can find a constant $C \in \R$ such that $|\Ch_n| \leq C \frac{\Gamma(n -\frac12 )}{\log^2 (n+1)}$ for all $n \geq 1$.
    Recall that $\Gamma(x)$ is \textit{log-convex}  on the interval $x\in(0,\infty)$, i.e.\ $\log(\Gamma(x))$ is a convex function on this interval \cite{artin2015gamma}. The function $-\log(\log(1+x))$ is convex on this interval as well, since its second derivative $\frac{1+\log(1+x)}{(1+x)^2 \log^2(1+x)}$ is positive. If $f(x)$ is convex on the interval $[a,b]$, then $f(b+a-x)$ is also convex on $[a,b]$. If another function $g(x)$ is convex on this interval, then $f(x) + g(x)$ is too. Therefore, 
    \begin{align*}
        \log( \Gamma(n-x-\frac12) ) +  \log( \Gamma(x-\frac12) ) - 2 \log \log(1+n-x) - 2 \log \log(1+x)
    \end{align*}
    is convex for $x\in(\frac12 , n-\frac12)$. Because $e^x$ is an increasing function,
        $\frac{\Gamma(n-x-\frac12)\Gamma(x-\frac12)}{ \log^2(1+n-x) \log^2(1+x) }$
        is also convex on $x\in(\frac12, n-\frac12)$. This also implies convexity on the smaller interval $[2,n-2]$. The usual inequality for convex functions now gives
        \begin{align*}
            \frac{\Gamma(n-x-\frac12)\Gamma(x-\frac12)}{ \log^2(1+n-x) \log^2(1+x) }
            \leq \frac{\Gamma(n-2-\frac12)\Gamma(2-\frac12)}{ \log^2(1+n-2) \log^2(1+2) } \text{ for all } x \in [2,n-2],
        \end{align*}
        and we can estimate
        \begin{gather*}
            \left|\sum_{k=1}^{n-1} \Ch_{n-k} \Ch_k \right| \leq 2 | \Ch_{n-1} \Ch_1 | +\sum_{k=2}^{n-2} | \Ch_{n-k} \Ch_k| 
            \leq 2 | \Ch_{n-1} \Ch_1 | + C^2 \sum_{k=2}^{n-2} \frac{\Gamma(n-k -\frac12 )\Gamma(k -\frac12 )}{\log^2 (1+n-k)\log^2 (1+k)} \\
            \leq 2 | \Ch_{n-1} \Ch_1 | + C^2 (n-3) \frac{\Gamma(n-2 -\frac12 )\Gamma(2 -\frac12 )}{\log^2 (1+n-2)\log^2 (1+2)}.
        \end{gather*}
        It follows that $\sum_{k=1}^{n-1} \Ch_{n-k} \Ch_k \in \bigO(\Ch_{n-1})$, so Lemma~\ref{lmm:wright_connected_asymptotics} can be applied to give 
\begin{align*}
    \ch_n &= \Ch_n + \bigO(\Ch_{n-1}) = - \frac{1}{\sqrt{2\pi}} \frac{\Gamma(n -\frac12 )}{\log^2 n}  + \bigO\left( \frac{\log \log n }{\log^3 n}\Gamma\left(n -\frac12 \right) \right),
\end{align*}
because $\Ch_{n-1} \in \bigO\left( \frac{\log \log n }{\log^3 n}\Gamma\left(n -\frac12 \right) \right)$.
\end{proof}

The asymptotic behavior of $\chi(\Out(F_n))$ now follows by combining our results.
\begin{repthmx}{thm:SVconj}
    The rational Euler characteristic of $\Out(F_n)$ is strictly negative, $\chi\left( \Out(F_n) \right) < 0$, for all $n \geq 2$ and its magnitude grows more than exponentially,
    \begin{align*}
        \chi\left(\Out(F_{n}) \right) &\sim - \frac{1}{\sqrt{2 \pi}} \frac{\Gamma(n-\frac32)}{\log^2(n)}.
    \end{align*}
\end{repthmx}

\begin{proof}
    Apply Corollary~\ref{crll:negative_T} and Proposition~\ref{prop:chi_outfn_asymptotics}.
\end{proof}

\appendix

\section{Proof of Laplace's method}
\begin{replemma}{lmm:laplace_method}[Laplace's method]
Let $f$ and $g$ be real-valued functions  on a domain  $D \subset \R$  with $0$ in its interior.  Suppose both $f$ and $g$ are analytic in a neighborhood of $0$, that $g(0)=g'(0)=0$, $g''(0)=-1,$ and $0$ is the unique global supremum of $g$. Finally, assume that the integral
    \begin{align*}
 \int_D | f(x) | e^{n g(x)} dx
    \end{align*}
    exists for sufficiently large $n$.
  Then the sequence $I(n)$ given by the integral formula
\begin{align}
    \label{epn:integral_IA}
    I(n) = \sqrt{\frac{n}{2 \pi}} \int_D f(x) e^{n g(x)} dx
\end{align}
admits an asymptotic expansion with asymptotic scale $\{n^{-k}\}_{k\geq0}$, 
\begin{align}
    \label{epn:laplace_expansionA}
    I(n) \sim \sum_{k \geq 0} c_k n^{-k} \text{ as } n \rightarrow \infty,
\end{align}
where $c_k$ is the coefficient of $z^k$ in the formal power series,
\begin{align}
\label{epn:laplace_coeffsA}
 \sum_{\ell \geq 0} z^\ell (2\ell-1)!! [x^{2\ell}] f(x) e^{\frac{1}{z} \left( g(x) + \frac{x^2}{2} \right) }.
\end{align}
\end{replemma}

\newcommand{\floor}[1]{\lfloor #1 \rfloor}
\begin{proof}
The proof closely follows the arguments in \cite[Sec.\ 4.4]{de1981asymptotic} and \cite[Thm. B7]{flajolet2009analytic}. 

The dominating contribution to the value of the integral $I(n)$ for large $n$ comes from the values of $f$ and $g$ near the global supremum of $g$. Laplace's method works by exploiting this observation quantitatively. 
    
    The start of the method is to show that $I(n)$ can be approximated by integrating the same integrand over a small neighborhood of the supremum. The second step is to split the integrand into a product of a Gaussian kernel and an analytic function which can be expanded as a convergent power series. In the last step, this power series is then integrated termwise using the Gaussian integrals
\begin{gather}
        \label{eqn:gaussian_integral}
        \sqrt{\frac{n}{2 \pi}} \int_{-\infty}^{-\infty} e^{-n \frac{x^{2}}{2}} x^{2k} dx = n^{-k}(2k-1)!! \text{ and }
        \sqrt{\frac{n}{2 \pi}} \int_{-\infty}^{-\infty} e^{-n \frac{x^{2}}{2}} x^{2k+1} dx = 0,
\end{gather}
where both identities hold for all $k\geq 0$.
The resulting series is an asymptotic expansion of the integral $I(n)$.

We will start by showing that the most important contribution to the value of $I(n)$ comes from a small neighbourhood of the supremum of $g$.
Because $x=0$ is a local maximum of $g(x)$ in $D$, the function $g(x)$ is monotonically increasing on the interval $[-\delta,0]$ and monotonically decreasing on $[0,\delta]$ for a sufficiently small $\delta > 0$. Because $x=0$ is the unique global supremum of $g(x)$ in $D$, we can also choose $\delta$ small enough such that if $\epsilon \in [0,\delta]$, then $g(x) \leq \max(g(\epsilon),g(-\epsilon))$ for all $x \in D \setminus (-\epsilon,\epsilon)$. This observation translates into the following estimate
\begin{gather}
    \begin{gathered}
    \label{eqn:laplace_epsilon_bound_I}
\left| I(n) - \sqrt{\frac{n}{2 \pi}}\int_{-\epsilon}^\epsilon f(x) e^{n g(x)} dx \right|
\leq
\sqrt{\frac{n}{2 \pi}} 
\int_{D \setminus (-\epsilon,\epsilon)}  |f(x)| e^{(n-n_0+n_0) g(x)}dx 
\\
\leq
\sqrt{\frac{n}{2 \pi}} e^{(n-n_0)\max(g(-\epsilon),g(\epsilon))} 
\int_{D}  |f(x)| e^{n_0 g(x)}dx,
\end{gathered}
\end{gather}
where $n_0$ is chosen appropriately large such that, in accordance with the requirement, the integral in the last bound exists. We may assume that $\delta$ is small enough such that $f$ and $g$ are analytic in the interval $[-\delta,\delta]$. Because of the requirements $g(0)=g'(0)=0$ and $g''(0)=-1$, the function $x^{-3} \left( g(x) + \frac{x^2}{2} \right)$ is analytic in this interval as well. Specifically, it is bounded in $[-\delta, \delta]$. Therefore, there exists a constant $C_1 \in \R$ such that $|\epsilon|^{-3} \left| g(\epsilon) + \frac{\epsilon^2}{2} \right| \leq C_1$ for $\epsilon \in [-\delta,\delta]$ or equivalently
$\left| g(\pm \epsilon) + \frac{\epsilon^2}{2} \right| \leq C_1 \epsilon^3$ for $\epsilon \in [0,\delta]$. Therefore,
\begin{gather*}
    (n -n_0)g(\pm \epsilon) = 
-\frac{\epsilon^2}{2}
+
n \left( g( \pm \epsilon) + \frac{\epsilon^2}{2} \right)
- n_0 g(\pm \epsilon)
\leq
-\frac{\epsilon^2}{2}
+ C_1 \epsilon^3 - n_0 g(\pm \delta) \text{ for } \epsilon \in [0, \delta],
\end{gather*}
where $- n_0 g(\pm \epsilon) \leq - n_0 g(\pm \delta)$ follows from the monotonicity of $g$ on the intervals $[-\delta,0]$ and $[0,\delta]$. If we set $\epsilon = n^{-{\frac{5}{12}}}$, then for large enough $n$
\begin{gather*}
    e^{(n-n_0)g(\pm n^{-{\frac{5}{12}}})} \leq \exp \left( 
-\frac{n^{{\frac{1}{6}}}}{2}
+ C_1 n^{-{\frac{1}{4}}} - n_0 g(\pm \delta)
    \right) \in \bigO(e^{-\frac{n^{{\frac{1}{6}}}}{2}}).
\end{gather*}
Applying the inequality in eq.\ \eqref{eqn:laplace_epsilon_bound_I}, where we also set $\epsilon = n^{-{\frac{5}{12}}}$, gives, 
\begin{align}
    \label{eqn:I_truncated}
    I(n) &= \sqrt{\frac{n}{2 \pi}}\int_{-n^{-{\frac{5}{12}}}}^{n^{-{\frac{5}{12}}}} f(x) e^{n g(x)} dx + \bigO(\sqrt{n}e^{-\frac{n^{{\frac{1}{6}}}}{2}}),
\end{align}
where the $\bigO$ refers, as in the rest of the proof, to the limit $n \rightarrow \infty$.
Observe that we managed to approximate $I(n)$ up to an exponentially small contribution without using the information of the domain $D$. Instead of $\epsilon = n^{-{\frac{5}{12}}}$, we also could have chosen $\epsilon = n^{-\gamma}$ with any $\frac13 < \gamma < \frac12$.

The second step of Laplace's method is to interpret the integrand as a Gaussian kernel times the function
\begin{align*}
A(x,n) := f(x) e^{n \left( g(x) + \frac{x^2}{2} \right) }, 
\end{align*}
which we may expand as a power series, because it is analytic in $x$ and entire in $n$,
\begin{align*}
    A(x,n) &= \sum_{k=0}^{\infty} \sum_{\ell=0}^{\floor{\frac{k}{3}}} a_{k,\ell} n^\ell x^k.
\end{align*}
This expansion is convergent for all $x \in [-\delta,\delta]$ if we choose $\delta$ sufficiently small.
The sum over $\ell$ is bounded, because $g(x) + \frac{x^2}{2} = \alpha x^3 + \text{`higher order terms'}$. As $A(x,n)$ is analytic in $x$ and $n$, there exists a constant $C_2\in \R$ such that $|a_{k,\ell}| \leq C_2^{k+\ell}$ for all $k,\ell \geq 0$. We will truncate the power series expansion of $A$ at some finite order $K \geq 0$ in $x$. The remainder can be estimated uniformly for $n$ sufficiently large and $x\in [-n^{-{\frac{5}{12}}},n^{-{\frac{5}{12}}}]$:
\begin{gather*}
    \left| f(x) e^{n \left( g(x) + \frac{x^2}{2} \right)} - \sum_{k=0}^{K-1} \sum_{\ell=0}^{\floor{\frac{k}{3}}} a_{k,\ell} n^\ell x^k \right| = 
    \left| \sum_{k=K}^{\infty} \sum_{\ell=0}^{\floor{\frac{k}{3}}} a_{k,\ell} n^\ell x^k \right|
    \leq
    |x|^{K} \sum_{k=0}^{\infty} \sum_{\ell=0}^{\floor{\frac{k+K}{3}}} C_2^{k+K+\ell} n^\ell |x|^k 
    \\
    \leq
    |x|^{K} \sum_{k=0}^{\infty} \frac{k+K}{3} C_2^{k+K+\frac{k+K}{3}} n^{\frac{k+K}{3}} n^{-\frac{5}{12}k} 
    =
    |x|^{K} C_2^{\frac{4}{3}K} n^{\frac{K}{3}}
    \sum_{k=0}^{\infty} \frac{k+K}{3} \left( \frac{n}{C_2^{16}}\right)^{-\frac{1}{12}k} \leq C_3 n^{\frac{K}{3}} |x|^K.
\end{gather*}
This estimate only makes sense when $n > C_2^{16}$, which we may require. The constant $C_3$ is chosen appropriately independent of $x$ and $n$.
With $K=6R$, we get the following estimate for integrals from this, 
\begin{gather*}
    \sqrt{\frac{n}{2 \pi}}\left| \int_{-n^{-{\frac{5}{12}}}}^{n^{-{\frac{5}{12}}}} \left( f(x) e^{n g(x)} 
-
 e^{-n \frac{x^{2}}{2}} \sum_{k=0}^{6R-1} \sum_{\ell=0}^{\floor{\frac{k}{3}}} a_{k,\ell} n^\ell x^k \right) dx
\right| \leq 
    C_3 n^{2 R} \sqrt{\frac{n}{2 \pi}}\int_{-n^{-{\frac{5}{12}}}}^{n^{-{\frac{5}{12}}}} e^{-n \frac{x^{2}}{2}} x^{6R} dx
    \\
    \leq
    C_3 n^{2 R} \sqrt{\frac{n}{2 \pi}} \int_{-\infty}^{-\infty} e^{-n \frac{x^{2}}{2}} x^{6R} dx
    \leq 
    C_3 (6R-1)!! n^{-R} \text{ for all } R \geq 0,
\end{gather*}
where the Gaussian integral from eq.\ \eqref{eqn:gaussian_integral} was used in the last step.
Together with eq.\ \eqref{eqn:I_truncated} and $\bigO(\sqrt{n} e^{-n^{1/6} /2}) \subset \bigO(n^{-R})$ this implies that 
\begin{align}
    \label{eqn:I_expanded}
    I(n) &=
\sum_{k=0}^{6R-1}
 \sum_{\ell=0}^{\floor{\frac{k}{3}}} a_{k,\ell} n^\ell \sqrt{\frac{n}{2 \pi}} 
\int_{-n^{-{\frac{5}{12}}}}^{n^{-{\frac{5}{12}}}} 
e^{-n \frac{x^{2}}{2}} x^k dx
+ \bigO(n^{-R}) \text{ for all } R \geq 0.
\end{align}
The remaining task is to get rid of the finite integration bounds to recover a full Gaussian integral. We need another bound for the remainder,
\begin{gather*}
    r_k(n,\epsilon) :=
\sqrt{\frac{n}{2 \pi}} \left| \int_{-\infty}^{\infty} e^{-n \frac{x^2}{2} } x^{2k} dx - 
\int_{-\epsilon}^{\epsilon} e^{-n \frac{x^2}{2} } x^{2k} dx \right| \\
= 
\sqrt{\frac{n}{2 \pi}}
\left(
\int_{-\infty}^{-\epsilon} e^{-n \frac{x^2}{2} } x^{2k} dx
+
\int_{\epsilon}^{\infty} e^{-n \frac{x^2}{2} } x^{2k} dx
\right)
\\
=
2\sqrt{\frac{n}{2 \pi}}
\int_{0}^{\infty} e^{-n \frac{\left(x+\epsilon\right)^2}{2} } \left(x+\epsilon\right)^{2k} dx
=
2e^{-n\frac{\epsilon^2}{2}}
\sqrt{\frac{n}{2 \pi}}
\int_{0}^{\infty} e^{-n \frac{x^2}{2} - \epsilon n x } \left(x+\epsilon\right)^{2k} dx.
\end{gather*}
The function $e^{- \epsilon n x } \left(x+\epsilon\right)^{2k}$ is bounded for $x\in [-\epsilon,\infty]$ and has its right most local maximum at $x = \frac{2k}{\epsilon n} - \epsilon$. As long as $2k \leq \epsilon^2 n$, this maximum lies on or on the left of the origin and the function decreases monotonically for $x \in [0,\infty)$. Therefore,
\begin{align*}
    r_k(n,\epsilon) \leq 
2e^{-n\frac{\epsilon^2}{2}}
\epsilon^{2k}
\sqrt{\frac{n}{2 \pi}}
\int_{0}^{\infty} e^{-n \frac{x^2}{2}} dx 
=
e^{-n\frac{\epsilon^2}{2}}
\epsilon^{2k} \text{ for all } \epsilon, n \text{ and } k \text{ such that } 2k \leq \epsilon^2 n. 
\end{align*}
Specifically, with $\epsilon=n^{-{\frac{5}{12}}}$, we get $r_k(n,n^{-{\frac{5}{12}}}) \in \bigO(e^{-\frac{1}{2}n^{{\frac{1}{6}}}}) \subset \bigO(n^{-R})$.
Applying this to eq.~\eqref{eqn:I_expanded} finally gives,
\begin{align*}
    I(n) &=
\sum_{k=0}^{6R-1}
 \sum_{\ell=0}^{\floor{\frac{k}{3}}} a_{k,\ell} n^\ell 
\sqrt{\frac{n}{2 \pi}} 
\int_{-\infty}^{\infty} 
e^{-n \frac{x^{2}}{2}} x^k dx
+ \bigO(n^{-R}) \text{ for all } R \geq 0,
\end{align*}
which can be rewritten as eqs.\ \eqref{epn:laplace_expansionA} and \eqref{epn:laplace_coeffsA}, by substituting the definition of $A(x,n)$, by using the Gaussian integral from eq.\ \eqref{eqn:gaussian_integral}, by realizing that only every second term in the sum over $k$ contributes, by using the notation from Definition \ref{def:asymptotic_expansion} and the coefficient extraction operator.
\end{proof}


\begin{thebibliography}{10}

\bibitem{artin2015gamma}
E.~Artin.
\newblock {\em The gamma function}.
\newblock Holt, Rinehart and Winston, 1964.

\bibitem{Bartholdi}
L.~Bartholdi.
\newblock The rational homology of the outer automorphism group of {$F_7$}.
\newblock {\em New York J. Math.}, 22:191--197, 2016.

\bibitem{bessis1980quantum}
D.~Bessis, C.~Itzykson, and J.-B. Zuber.
\newblock Quantum field theory techniques in graphical enumeration.
\newblock {\em Adv. in Applied Math.}, 1(2):109--157, 1980.

\bibitem{BBM}
M.~Bestvina, K.-U. Bux, and D.~Margalit.
\newblock Dimension of the {Torelli} group for {$\Out(F_n)$}.
\newblock {\em Inventiones Mathematicae}, 170(1):1--32, 2007.

\bibitem{BF00}
M.~Bestvina and M.~Feighn.
\newblock {The topology at infinity of $\Outn$}.
\newblock {\em Inventiones Mathematicae}, 140(3):651--692, 2000.

\bibitem{borinsky2017renormalized}
M.~Borinsky.
\newblock Renormalized asymptotic enumeration of {Feynman} diagrams.
\newblock {\em Annals of Phys.}, 385:95--135, 2017.

\bibitem{borinsky2018generating}
M.~Borinsky.
\newblock Generating asymptotics for factorially divergent sequences.
\newblock {\em The Electronic J. of Combinatorics}, 25(4):4--1, 2018.

\bibitem{borinsky2018graphs}
M.~Borinsky.
\newblock {\em Graphs in perturbation theory: algebraic structure and
  asymptotics}.
\newblock Springer, 2018.

\bibitem{borwein2018gamma}
J.~M. Borwein and R.~M. Corless.
\newblock Gamma and factorial in the {Monthly}.
\newblock {\em The American Mathematical Monthly}, 125(5):400--424, 2018.

\bibitem{Bro82}
K.~S. Brown.
\newblock {\em Cohomology of groups}.
\newblock Springer, Berlin, 1982.

\bibitem{Brown82}
K.~S. Brown.
\newblock Complete {E}uler characteristics and fixed-point theory.
\newblock {\em J. Pure Appl. Algebra}, 24(2):103--121, 1982.

\bibitem{BSV18}
K.-U. Bux, P.~Smillie, and K.~Vogtmann.
\newblock {On the bordification of Outer space}.
\newblock {\em J. London Math. Soc.}, 98:12--34, 2018.

\bibitem{CHKV16}
J.~Conant, A.~Hatcher, M.~Kassabov, and K.~Vogtmann.
\newblock {Assembling homology classes in automorphism groups of free groups}.
\newblock {\em Commentarii Mathematici Helvetici}, 91(4):751--806, 2016.

\bibitem{conant2003theorem}
J.~Conant and K.~Vogtmann.
\newblock On a theorem of {Kontsevich}.
\newblock {\em Algebr. Geom. Topol.}, 3(2):1167--1224, 2003.

\bibitem{connes2000renormalization}
A.~Connes and D.~Kreimer.
\newblock Renormalization in quantum field theory and the {Riemann--Hilbert
  problem I}: The {Hopf} algebra structure of graphs and the main theorem.
\newblock {\em Comm. Math. Phys.}, 210(1):249--273, 2000.

\bibitem{corless1996lambertw}
R.~M. Corless, G.~H. Gonnet, D.~E. Hare, D.~J. Jeffrey, and D.~E. Knuth.
\newblock On the {Lambert-$W$} function.
\newblock {\em Adv. in Comp. Math.}, 5(1):329--359, 1996.

\bibitem{CV86}
M.~Culler and K.~Vogtmann.
\newblock {Moduli of graphs and automorphisms of free groups}.
\newblock {\em Inventiones Mathematicae}, 84(1):91--119, 1986.

\bibitem{de1981asymptotic}
N.~G. de~Bruijn.
\newblock {\em Asymptotic methods in analysis}.
\newblock North-Holland Publishing, 1958.

\bibitem{flajolet1990singularity}
P.~Flajolet and A.~Odlyzko.
\newblock Singularity analysis of generating functions.
\newblock {\em SIAM J. on Discrete Math.}, 3(2):216--240, 1990.

\bibitem{flajolet2009analytic}
P.~Flajolet and R.~Sedgewick.
\newblock {\em Analytic combinatorics}.
\newblock Cambridge University Press, 2009.

\bibitem{Galatius}
S.~Galatius.
\newblock Stable homology of automorphism groups of free groups.
\newblock {\em Annals of Math.}, 173(2):705--768, 2011.

\bibitem{gerlits2002}
F.~Gerlits.
\newblock {\em Invariants in chain complexes of graphs}.
\newblock ProQuest LLC, Ann Arbor, MI, 2002.
\newblock Thesis (Ph.D.)--Cornell University.

\bibitem{HZ86}
J.~L. Harer and D.~Zagier.
\newblock {The Euler characteristic of the moduli space of curves}.
\newblock {\em Inventiones Mathematicae}, 85(3):457--485, 1986.

\bibitem{Hat95}
A.~Hatcher.
\newblock {Homological stability for automorphism groups of free groups}.
\newblock {\em Commentarii Mathematici Helvetici}, 70(1):39--62, 1995.

\bibitem{HV96}
A.~Hatcher and K.~Vogtmann.
\newblock Isoperimetric inequalities for automorphism groups of free groups.
\newblock {\em Pacific J. Math.}, 173(2):425--441, 1996.

\bibitem{HV98}
A.~Hatcher and K.~Vogtmann.
\newblock Rational homology of {${\rm Aut}(F_n)$}.
\newblock {\em Math. Res. Lett.}, 5(6):759--780, 1998.

\bibitem{Kon92}
M.~Kontsevich.
\newblock Intersection theory on the moduli space of curves and the matrix
  {Airy} function.
\newblock {\em Comm. Math. Phys.}, 147(1):1--23, 1992.

\bibitem{kontsevich1993formal}
M.~Kontsevich.
\newblock Formal (non)-commutative symplectic geometry.
\newblock In {\em The Gelfand Mathematical Seminars, 1990--1992}, pages
  173--187. Springer, 1993.

\bibitem{kontsevich1994feynman}
M.~Kontsevich.
\newblock Feynman diagrams and low-dimensional topology.
\newblock In {\em First European Congress of Mathematics Paris, July 6--10,
  1992}, pages 97--121. Springer, 1994.

\bibitem{kreimer2009core}
D.~Kreimer.
\newblock The core {Hopf} algebra.
\newblock {\em Clay Math. Proc.}, 11:313--322, 2009.

\bibitem{KrVo93}
S.~Krsti\'{c} and K.~Vogtmann.
\newblock Equivariant outer space and automorphisms of free-by-finite groups.
\newblock {\em Comment. Math. Helv.}, 68(2):216--262, 1993.

\bibitem{Magnus}
W.~Magnus.
\newblock {{\"U}ber $n$-dimensionale Gittertransformationen}.
\newblock {\em Acta Mathematica}, 64(1):353--367, Dec. 1935.

\bibitem{penner1986moduli}
R.~C. Penner.
\newblock The moduli space of a punctured surface and perturbative series.
\newblock {\em Bulletin of the American Math. Soc.}, 15(1):73--77, 1986.

\bibitem{Serre71}
J.-P. Serre.
\newblock {Cohomologie des groupes discrets}.
\newblock {\em Annals of Math. Studies}, 70:77--169, 1971.

\bibitem{Serre79}
J.-P. Serre.
\newblock {\em Arithmetic Groups}, pages 105--136.
\newblock Number~36 in London Math. Soc. Lecture Notes. Cambridge University
  Press, 1979.

\bibitem{SV87a}
J.~Smillie and K.~Vogtmann.
\newblock {A generating function for the Euler characteristic of $\Outn$}.
\newblock {\em J. of Pure and Applied Algebra}, 44(1-3):329--348, 1987.

\bibitem{SV87b}
J.~Smillie and K.~Vogtmann.
\newblock {Automorphisms of graphs, p-subgroups of $\Outn$ and the Euler
  characteristics of $\Outn$}.
\newblock {\em J. of Pure and Applied Algebra}, 1987.

\bibitem{stanley1997enumerative2}
R.~P. Stanley.
\newblock {\em Enumerative Combinatorics}.
\newblock Number Volume 2 in Cambridge Studies in Adv. Math. Cambridge
  University Press, 1997.

\bibitem{sweedler1969hopf}
M.~E. Sweedler.
\newblock {\em Hopf algebras}.
\newblock Mathematics Lecture Note Series. W. A. Benjamin, 1969.

\bibitem{volkmer2008factorial}
H.~Volkmer.
\newblock Factorial series connected with the {Lambert} function, and a problem
  posed by {Ramanujan}.
\newblock {\em The Ramanujan Journal}, 16(3):235--245, 2008.

\bibitem{Wall61}
C.~T.~C. Wall.
\newblock Rational {E}uler characteristics.
\newblock {\em Proc. Cambridge Philos. Soc.}, 57:182--184, 1961.

\bibitem{witten1988topological}
E.~Witten.
\newblock Topological quantum field theory.
\newblock {\em Comm. Math. Phys.}, 117(3):353--386, 1988.

\bibitem{witten1990two}
E.~Witten.
\newblock Two-dimensional gravity and intersection theory on moduli space.
\newblock {\em Surveys in differential geometry}, 1(1):243--310, 1990.

\bibitem{wright1970asymptotic}
E.~M. Wright.
\newblock Asymptotic relations between enumerative functions in graph theory.
\newblock {\em Proc. of the London Math. Soc.}, 3(3):558--572, 1970.

\bibitem{Zagier}
D.~Zagier, 1989.
\newblock personal communication.

\end{thebibliography}
\end{document}